\documentclass{amsart}
\usepackage[margin=2.5cm,marginpar=2cm]{geometry}
\usepackage{amssymb,url}
\usepackage[mathscr]{eucal}
\usepackage{mystyle}
\usepackage{enumitem}

\begin{document}

\title{Linear recurrence relations in $Q$-systems via lattice points in polyhedra}
\author{Chul-hee Lee}
\address{School of Mathematics and Physics, The University of Queensland, Brisbane QLD 4072, Australia}
\email{c.lee1@uq.edu.au}
\date{\today}
\thanks{This work was supported by the Australian Research Council.}
\begin{abstract}
We prove that the sequence of the characters of the Kirillov-Reshetikhin (KR) modules $W_{m}^{(a)}, m\in \mathbb{Z}_{m\geq 0}$ associated to a node $a$ of the Dynkin diagram of a complex simple Lie algebra $\mathfrak{g}$ satisfies a linear recurrence relation except for some cases in types $E_7$ and $E_8$. To this end we use the $Q$-system and the existing lattice point summation formula for the decomposition of KR modules, known as domino removal rules when $\mathfrak{g}$ is of classical type. As an application, we show how to reduce some unproven lattice point summation formulas in exceptional types to finite problems in linear algebra and also give a new proof of them in type $G_2$, which is the only completely proven case when KR modules have an irreducible summand with multiplicity greater than 1. We also apply the recurrence to prove that the function $\dim W_{m}^{(a)}$ is a quasipolynomial in $m$ and establish its properties. We conjecture that there exists a rational polytope such that its Ehrhart quasipolynomial in $m$ is $\dim W_{m}^{(a)}$ and the lattice points of its $m$-th dilate carry the same crystal structure as the crystal associated with $W_{m}^{(a)}$.
\end{abstract}
\maketitle
\section{Introduction}
Let $\mathfrak{g}$ be a complex simple Lie algebra. The decomposition of Kirillov-Reshetikhin (KR) modules over $U_q(\widehat{\mathfrak{g}})$ or their tensor products into irreducible $U_q(\mathfrak{g})$-modules is given by the fermionic formula, originally conjectured in \cite{Kirillov1990}. It was obtained by transferring insight from the Bethe ansatz in mathematical physics to representation theory and is interpreted as the combinatorial completeness of the Bethe ansatz.

The strategy of the proof of the fermionic formula, proposed in Theorem 8.1 of \cite{MR1745263}, consists mainly of three steps. The first is to show that the family $\{Q_m^{(a)}\}_{a\in I,m\in \mathbb{Z}_{\geq 0}}$ of the characters of the KR modules $\res W_m^{(a)}$ as $U_q(\mathfrak{g})$-modules satisfies a certain set of difference equations called the $Q$-system \cite{MR2254805, MR1993360}. Here $I$ denotes the set of nodes of the Dynkin diagram of $\mathfrak{g}$. The next step is to establish asymptotics of the ratio $\frac{Q_{m+1}^{(a)}}{Q_m^{(a)}}$ as $m$ goes to infinity \cite{MR2254805}. The last one is to prove the ``$M=N$ conjecture" \cite{MR2428305}. These works combined altogether give a proof of the Kirillov-Reshetikhin conjecture on the fermionic formula.

The formula expresses the multiplicity of an irreducible summand in a certain combinatorial way. In spite of its elegance, it quickly becomes impractical as the rank of $\mathfrak{g}$ increases due to its complicated combinatorial nature. Although there is an algorithm which makes this formula more efficient \cite{MR1436775}, it is still advantageous to have a more explicit description of the decomposition of KR modules for practical purposes. In \cite{MR1745263, MR1436775}, such an alternative formula was conjecturally given for most of KR modules even when $\mathfrak{g}$ is of exceptional type. In \cite{MR1836791} a proof of it was obtained when the multiplicity of each irreducible summand inside $\res W_m^{(a)}$ is always 1. This type of formula has been traditionally known as rules for removals of dominoes from a rectangle when $\mathfrak{g}$ is of classical type; see, for example, Section 13.7 of \cite{MR2773889}. In this paper we shall call it a \textit{lattice point summation formula} as the highest weights of the irreducible summands with non-zero multiplicity can be characterized as lattice points in a suitable rational polyhedron; see Section \ref{sec:lpsf}.

The asymptotic behavior of the sequence $\{Q_m^{(a)}\}_{m=0}^{\infty}$ for fixed $a\in I$ is closely related to the existence of linear recurrence relations satisfied by it, which we will study in this paper. Although these relations in solutions of $Q$-systems seem to be quite natural to consider in view of the importance of its asymptotics, the treatment of the topic in the general setting has not been given until quite recently. They have been studied in special cases in several different contexts. For example, those relations in type $A$ are considered in \cite{dfk} from the viewpoint of integrable systems. They are also studied in \cite{2009arXiv0905.3776N} for type $A$ and $D$, motivated from a problem in number theory. In \cite{MR3279992} a gauge theoretic discussion on the topic appeared. In \cite{MR3342762} we gave some general conjectures on the form of these relations which have arisen from experiments and proved them in some special cases by using the known results on difference $L$-operators for the $q$-character solutions of $T$-systems \cite{MR1890924}. We also add that linear recurrence relations in a certain family of polynomials are studied in \cite{MR3484752} and some of these results could be used to handle our problem in type $A$ if one uses the known relationship between KR modules and affine Demazure modules \cite{MR2235341,MR2323538}.

It is also essential to know about the detailed structure of these relations when studying the behavior of the characters of KR modules under various specializations. This is the case when we try to determine the images of KR modules in the fusion ring of affine integrable representations of positive integer level under a certain ring homomorphism defined on the Grothendieck ring of finite-dimensional representations of $U_q(\widehat{\mathfrak{g}})$. This is investigated in \cite{lee2013positivity} to answer a question originating in Kirillov's work on dilogarithm identities \cite{Kirillov1989}; see \cite[Section 14]{MR2773889} also for a broader context of the problem. There we observed some periodic phenomena (Conjecture 3.10 (iv) of \cite{lee2013positivity}), which can now be understood using the simplicity of roots of the characteristic polynomial of our linear recursion; see Theorem \ref{thm:main} \ref{thm:main:3}. 

Another specialization of the characters gives the dimensions of KR modules. The asymptotic polynomial growth of $\dim W^{(a)}_m$ in $m$ is known in simply-laced cases \cite{MR1436775}. We will show that this is actually a quasipolynomial in general and a polynomial in simply-laced cases. The properties of this quasipolynomial are consistent with the ones appearing in Ehrhart theory of convex polytopes; see \cite{MR3410115} for an introduction to the topic. It leads us to expect (Conjecture \ref{KRpolytope}) that there is a rational polytope $\mathcal{P}^{(a)}$ such that the lattice points of its $m$-th dilate $m \mathcal{P}^{(a)}$ has the same crystal structure as the crystal basis of $W_m^{(a)}$. The existence of KR crystals is known in classical types \cite{MR2553378, MR2995050, MR2403558} and remains conjectural in exceptional types. The polytopal construction of the KR crystals in type $A$ given in \cite{MR3102176} supports our claim. It seems to be closely related with the similarity of KR crystals \cite{MR3203905}; see \cite{MR1395599} as well. For more works discussing the role of polytopes in representation theory, see, for example, \cite{MR2834187, MR2785495}.

The main theme of this paper is to show that the following two topics mentioned above are very closely related:
\begin{enumerate}[label=(\roman*)]
\item linear recurrence relation satisfied by the sequence $\{Q^{(a)}_m\}_{m=0}^{\infty}$ of the characters of the KR modules for $a\in I$
\item lattice point summation formula for the decomposition of the KR modules associated to $a\in I$
\end{enumerate}

We will first show that the sequence $\{Q^{(a)}_m\}_{m=0}^{\infty}$ satisfies a linear recurrence relation for all $a\in I$ except for some cases in type $E_7$ and $E_8$; see Theorem \ref{thm:main} for the precise statement. In Section \ref{sec:review}, we give a review of background results including the fermionic formula and lattice point summation formulas. In Section \ref{sec:proof}, we give a proof of Theorem \ref{thm:main}. The central idea is that the existing lattice point summation formula for $a\in I$ gives a linear recurrence relation in $\{Q^{(a)}_m\}_{m=0}^{\infty}$. Once this is done, we can also show such a relation for a different node in the Dynkin diagram by employing the $Q$-system.

In Sections \ref{sec:app} and \ref{sec:app2}, we give applications of Theorem \ref{thm:main}. In Section \ref{sec:app}, we turn our focus to conjectural lattice point summation formulas when $\res W_m^{(a)}$ has an irreducible summand with multiplicity greater than 1. The main idea is that once a linear recurrence relation for $\{Q^{(a)}_m\}_{m=0}^{\infty}$ is established, it is sometimes possible to reduce problems about it to problems in linear algebra. As an illustration of this, we give a proof of some unproven lattice point summation formulas up to a finite amount of computation, which might be carried out using computer algebra systems; see Theorem \ref{thm:sumef}. The formula in type $G_2$ is so far the only example for which we have a proof \cite{MR2372556} when there is an irreducible summand occurring multiple times and we also give a new proof of this using our method; see Theorem \ref{thm:sumg2}. In Section \ref{sec:app2}, we study the dimension quasipolynomials of KR modules. For any $\mathfrak{g}$ and $a\in I$, we show that $\dim W^{(a)}_m$ is a quasipolynomial in $m$ and establish their properties as a quasipolynomial. In particular, we prove a certain reciprocity \eqref{recEM} satisfied by it, relating the values at negative integers to those at positive integers.

\subsection*{Notation}
Let $\mathfrak{g}$ be a complex simple Lie algebra of rank $r$ with a Cartan subalgebra $\mathfrak{h}$. Let $I=\{1,\dots, r\}$ be the set of nodes of its Dynkin diagram and $\{\alpha_a\}_{a\in I}$ be the set of simple roots. We will use the same ordering for $I$ as in \cite{MR3342762}.

Let $t_a=(\theta,\theta)/(\alpha_a,\alpha_a)\in \{1,2,3\}$ for each $a\in I$, where $\theta$ is the highest root and $(\cdot,\cdot)$ is the bilinear form on $\mathfrak{h}^{*}$ obtained from the Killing form, normalized by $(\theta,\theta)=2$. Note that $t_a=1$ if $\alpha_a$ is a long root. Let $C=(C_{ab})_{a,b\in I}$ be the Cartan matrix where $C_{ab}=(\alpha_a^{\vee},\alpha_b)$ and $\alpha_a^{\vee}=\frac{2}{(\alpha_a,\alpha_a)}\alpha_a$. We denote the fundamental weights by $\omega_a$ for each $a\in I$ and put $\omega_0=0$ for convenience. Let $Q=\oplus_{a\in I}\mathbb{Z}\alpha_a$ be the root lattice, $P=\oplus_{a\in I}\mathbb{Z}\omega_a$ the weight lattice and $P^+=\oplus_{a\in I}\mathbb{Z}_{\geq 0}\omega_a$ be the set of dominant integral weights. Let 
$\mathfrak{h}_{\mathbb{R}}^{*}:=\mathbb{R}\otimes P$. For $\lambda,\mu\in \mathfrak{h}_{\mathbb{R}}^{*}$, $\lambda\geq \mu$ or $\lambda\succeq \mu$ means $\lambda-\mu$ is in $\oplus_{a\in I}\mathbb{R}_{\geq 0}\alpha_a$ or $\oplus_{a\in I}\mathbb{Z}_{\geq 0}\alpha_a$, respectively.

We denote the Weyl group by $W$, which is generated by the set of simple reflections $\{s_a\}_{a\in I}$ on $\mathfrak{h}_{\mathbb{R}}^{*}$. Here each $\alpha_a$ is orthogonal to the reflecting hyperplane of $s_a$. For $w\in W$, we denote its length, the minimum number of simple reflections to express $w$, by $\ell(w)$. For any element $\mu\in \mathfrak{h}_{\mathbb{R}}^{*}$, there exist  $w\in W$ and $\lambda\in \mathfrak{h}_{\mathbb{R}}^{*}$ such that $\lambda\geq 0$ and $w(\lambda)=\mu$; such $\lambda$ exists uniquely.
For $\lambda\in \mathfrak{h}_{\mathbb{R}}^{*}$, we define $\mathcal{O}(\lambda)=\{w(\lambda)|w\in W\}$ and $\mathcal{O}_{e}(\lambda)=\{e^{\mu}|\mu\in \mathcal{O}(\lambda)\}$.
Note that a $W$-invariant finite subset $\Lambda$ of $P$ can be decomposed into a disjoint union of distinct $W$-orbits, say, $\Lambda=\coprod_{j\in J} \mathcal{O}(\lambda_j)$ with $\lambda_j\in P^{+}$. The size of $\mathcal{O}(\lambda),\lambda\in P^{+}$ can be easily computed as its isotropy subgroup is a standard parabolic subgroup, i.e., a subgroup generated by a subset of $\{s_a\}_{a\in I}$; see, for example, Theorem 1.12 in \cite{MR1066460}.

We put $y_a=e^{\omega_a}$ and denote the ring of Laurent polynomials in $\{y_{a}\}_{a\in I}$ by $\mathbb{Z}[P]$. As $W$ preserves $P$, it acts on $\mathbb{Z}[P]$ simply as $w(e^{\lambda}):=e^{w (\lambda)}$. In this paper we frequently work in the field $\mathbb{C}(y_1,\dots,y_r)$.
Note that the action of $W$ on $\mathbb{Z}[P]$ can be extended to this field and then each $w\in W$ is a field automorphism.

For a representation $V$ of $\mathfrak{g}$, we will denote its character by $\chi(V)\in \mathbb{Z}[P]$, which is an element of $\mathbb{Z}[P]^{W}$ and the set of its weights by $\Omega(V)$. An irreducible finite-dimensional representation of highest weight $\lambda\in P^{+}$ will be denoted by $L(\lambda)$. 

Let $q\in \mathbb{C}^{\times}$ be not a root of unity. As the theory of finite-dimensional representations for $U_q(\mathfrak{g})$ stays the same as that for $\mathfrak{g}$, we shall use the same notation as above for $U_q(\mathfrak{g})$.

Now let us introduce the basic terminology for linear recurrence relations. Let $F$ be a field and $\overline{F}$ its algebraic closure. A \textit{sequence} $\{a_n\}_{n=0}^{\infty}$ in $F$ is a function $a:\mathbb{Z}_{\geq 0}\to F$ whose value at $n$ is $a_n$. The \textit{shift operator} $\Delta$ acts on the space of sequences so that for a given sequence $\{a_n\}_{n=0}^{\infty}$, the sequence $\{(\Delta a)_n\}_{n=0}^{\infty}$ is given by
$$
(\Delta a)_n:=
\begin{cases}
a_{n-1},& \text{if $n\geq 1$}\\
0, & \text{if $n=0$}.
\end{cases}
$$
For simplicity of notation, we will also use $\Delta[a_n]:=(\Delta a)_n$. Given a polynomial $L(t)\in F[t]$ with $L(0)=1$, we call $\mathcal{L}:=L(\Delta)$ a \textit{difference operator}. We define the \textit{order} of $\mathcal{L}$ to be the degree of $L$ and call $\alpha\in \overline{F}^{\times}$ a \textit{root} of $\mathcal{L}$ if $L(\alpha^{-1})=0$.
If $L(t)=1+\sum_{k=1}^{\ell}c_k t^{k}$, then we have
$
\mathcal{L}[a_n]=a_n+\sum_{k=1}^{\ell}c_k a_{n-k}
$
for $n\geq \ell$. 
When
\begin{equation}\label{eq:lindiffeq}
\mathcal{L}[a_n]=0,\, \forall{n\geq N}
\end{equation}
holds for some $N\in \mathbb{Z}_{\geq 0}$, we say $\{a_n\}_{n=0}^{\infty}$ satisfies a \textit{linear recurrence relation of order $\ell$ with constant coefficients}. When \eqref{eq:lindiffeq} holds and $N$ in \eqref{eq:lindiffeq} can be taken to be $\ell$, we will simply write $\mathcal{L}[a_n]=0$.

\subsection*{Statement of the main theorem}
For each $a \in I$,  $m\in \mathbb{Z}_{\geq 0}$, and $u\in \mathbb{C}^{\times}$, there exists a finite-dimensional irreducible $U_{q}(\hat{\mathfrak{g}})$-module $W^{(a)}_{m}(u)$ called the Kirillov-Reshetikhin module; see \cite{MR1745263,MR2254805} for more detailed discussion. By restriction, we obtain a finite-dimensional $U_{q}(\mathfrak{g})$-module $\res \, W^{(a)}_m(u)$. As its dependence on $u$ disappears as $U_{q}(\mathfrak{g})$-module, we can simply write it as $\res \, W^{(a)}_m$. Let $Q_m^{(a)}=\chi(\res  W_m^{(a)})$ for each $a \in I$ and $m\in \mathbb{Z}_{\geq 0}$. For fixed $a\in I$, we obtain a sequence $\{Q_m^{(a)}\}_{m=0}^{\infty}$. 
\begin{theorem}\label{thm:main}
Let $\mathfrak{g}$ be a simple Lie algebra. Assume that $a\in I$ belongs to one of the following cases :
\begin{itemize}
\item $a\in I$ arbitrary when $\mathfrak{g}$ is of classical type or type $E_6$, $F_4$ or $G_2$, 
\item $a\in \{1,2,3,5,6,7\}$ when $\mathfrak{g}$ is of type $E_7$,
\item $a\in \{1,2,6,7\}$ when $\mathfrak{g}$ is of type $E_8$.
\end{itemize}
Then there exist finite subsets $\Lambda_a$ and $\Lambda'_a$ of $P$ with the following properties :
\begin{enumerate}[label=(\roman*), ref=(\roman*)]
\item \label{thm:main:2} $
\mathcal{L}_a[Q_{m}^{(a)}]=0
$
where
$$
\mathcal{L}_a=\prod_{\lambda \in \Lambda_a}(1-e^{\lambda}\Delta)\prod_{\lambda \in \Lambda'_a}(1-e^{\lambda}\Delta^{t_a}).
$$
\item \label{thm:main:1} $\Lambda_a$ and $\Lambda'_a$ are $W$-invariant.
\item \label{thm:main:3} $\mathcal{L}_a$ is without multiple roots, i.e., $t_a\Lambda_{a}\cap \Lambda'_{a}=\emptyset$, where $t_a\Lambda_{a}=\{t_a \lambda|\lambda \in \Lambda_a\}$.
\item \label{thm:main:4} $\omega_a\in \Lambda_a$.
\item \label{thm:main:6} For each $\lambda\in \Lambda_a$, $\omega_a\geq \lambda$ .
\item \label{thm:main:7} For each $\lambda\in \Lambda_a'$, $t_a\omega_a\geq \lambda$.
\end{enumerate}
We can take $\Lambda_a$ and $\Lambda'_a$ as given in Appendix \ref{sec:lambda}.
\end{theorem}

\begin{remark}
Let us write
$$
\sum_{k=0}^{\ell_a} C_{k}^{(a)}\Delta^k=\prod_{\lambda \in \Lambda_a}(1-e^{\lambda}\Delta)\prod_{\lambda \in \Lambda'_a}(1-e^{\lambda}\Delta^{t_a})
$$
where $\ell_a=|\Lambda_a|+t_a|\Lambda'_a|$. As $\Lambda_a$ and $\Lambda_a'$ are $W$-invariant, we have $C_{k}^{(a)}\in \mathbb{Z}[P]^W$ for each $k=0,\dots, \ell_a$. In particular, we get $C_{\ell_a}^{(a)}=(-1)^{|\Lambda_a|+|\Lambda'_a|}$ as $\Lambda_a$ and $\Lambda'_a$ are disjoint unions of $W$-orbits : for any $W$-orbit $\mathcal{O}(\lambda),\, \lambda\in P^+$ we have $\sum_{\mu\in \mathcal{O}(\lambda)}\mu=0$ as it is $W$-invariant and $0$ is the only $W$-invariant element in $P$.

The conditions \ref{thm:main:4},\ref{thm:main:6},\ref{thm:main:7} are related to 
\begin{equation}
\lim_{m\to\infty}\frac{Q_{m+1}^{(a)}}{Q_{m}^{(a)}}=e^{\omega_a}
\end{equation}
when the characters are evaluated at suitable elements of $\mathfrak{h}$ or $\mathfrak{h}^{*}$. This is the asymptotic behavior mentioned as one of the crucial steps in the proof of the fermionic formula.
\end{remark}
\begin{example}
In type $E_6$, the orders of the recurrences are 27, 27, 73, 243, 243, 1063, respectively. In \cite{MR3342762}, we were able to obtain some of these numbers experimentally but could not explain what they really were. Now we know, for example, that $1063 = 270 + 720 + 72 + 1$, where each of $270, 720, 72, 1$ is the size of the orbit of a certain weight under the Weyl group action. Understanding this structural property is extremely useful as it gives the precise bound when we apply the linear recurrence to reduce some unproven lattice point summation formulas to a finite verification, as described in Section 4.
\end{example}

\section{Background}\label{sec:review}
\subsection{Basic properties of linear recurrence relations}
Let $F$ be a field and $\overline{F}$ its algebraic closure. Assume that $\{a_n\}_{n=0}^{\infty}$ is a sequence in $F$. Let us first state two standard results on linear recurrence relations without proofs.
\begin{proposition}\label{prop:genL}
If the generating function of $\{a_n\}_{n=0}^{\infty}$ is a rational function
$$
\sum_{n=0}^{\infty} a_nt^n=\frac{b_0+b_1t+\cdots+b_{N}t^{N}}{1+c_1t+\cdots+c_{\ell}t^{\ell}}
$$
with $b_{N}\neq 0$ and $c_{\ell}\neq 0$, then $\mathcal{L}[a_n]=0$ for any $n\geq \max(N+1,\ell)$, where $\mathcal{L}=1+\sum_{k=1}^{\ell}c_k \Delta^k$. The converse also holds.
\end{proposition}

\begin{proposition}\label{prop:cr}
Let $\mathcal{L}$ be a difference operator without multiple roots and $A$ the set of its roots. If $\mathcal{L}[a_n]=0$, then there exists $d_{\alpha}\in \overline{F}$ for each $\alpha \in A$ such that
$
a_n=\sum_{\alpha\in A}d_{\alpha}\alpha^n
$
for all $n\in \mathbb{Z}_{\geq 0}$. The converse also holds.
\end{proposition}

In this paper we often consider a sequence $\{a_n'\}_{n=0}^{\infty}$ defined by
\begin{equation}\label{eq:derseq}
a_n':=a_n^2-a_{n+1}a_{n-1}.
\end{equation}
\begin{lemma}\label{lem:der}
Let $\mathcal{L}$ be a difference operator without multiple roots and $A$ the set of its roots. Let $\{a_n'\}_{n=0}^{\infty}$ be given as \eqref{eq:derseq} and $A'=\{\alpha\beta|\alpha,\beta \in A, \alpha\neq \beta \}$. If $\mathcal{L}[a_n]=0$, then we have
$\mathcal{L}'[a_n']=0
$
where $\mathcal{L}'=\prod_{\gamma\in A'}(1-\gamma t)$.
\end{lemma}
\begin{proof}
We can simply expand $a_n'=a_n^2-a_{n+1}a_{n-1}$ with
$a_n=\sum_{\alpha \in A} d_{\alpha}\alpha^{n}$
to get
$a'_n=\sum_{\gamma\in A'}d'_{\gamma}\gamma^n$, where $d_{\alpha}$ and $d'_{\gamma}$ are some elements of $\overline{F}$.
\end{proof}

\subsection{Kirillov-Reshetikhin modules and the fermionic formula}
It is known that the family $\left(Q^{(a)}_m\right)_{a\in I,m\in \mathbb{Z}_{\geq 0}}$ of the characters of the KR modules satisfies the following system of difference equations :
\begin{equation}\label{eq:Qsys}
(Q^{(a)}_m)^2 = Q^{(a)}_{m+1}Q^{(a)}_{m-1} + 
\prod_{b : C_{ab}< 0} \prod_{k=0}^{-C_{a b}-1}
Q^{(b)}_{\bigl \lfloor\frac{C_{b a}m - k}{C_{a b}}\bigr \rfloor}\quad a\in I, m \geq 1
\end{equation}
where $\lfloor \cdot \rfloor$ denotes the floor function. We call \eqref{eq:Qsys} the \textit{$Q$-system}. This was first introduced in \cite{Kirillov1990} and proved in \cite{MR2254805, MR1993360} as mentioned earlier in Introduction.

Let $\left(\nu^{(a)}_j\right)_{a\in I,j \geq 1}$ be a family of non-negative integers such that $\nu^{(a)}_j$'s are non-zero only for finitely many $(a,j)$. 
Let
\begin{equation}\label{eq:Wtensor}
W = \bigotimes_{a\in I} \bigotimes_{j \geq 1}\bigl(\res W^{(a)}_j\bigr)^{\otimes \nu^{(a)}_j}
\end{equation}
be a tensor product of Kirillov-Reshetikhin modules. For $\lambda \in P^{+}$ we define
$$
\begin{aligned}
M(W,\lambda) & = & \sum_{\mathbf{m}} 
\prod_{a\in I, i \geq 1}
\binom{p^{(a)}_i(\mathbf{m}) +  m^{(a)}_i}{m^{(a)}_i}
\end{aligned}
$$
where 
$$
\begin{aligned}
p^{(a)}_i(\mathbf{m})& = & \sum_{j \geq 1}\nu^{(a)}_j\mbox{min}(i,j)
- \sum_{b\in I} (\alpha_a , \alpha_b) \sum_{k \geq 1} 
\mbox{min}(t_bi,t_ak) m^{(b)}_k
\end{aligned}
$$
and the sum $\sum_{\mathbf{m}}$ in the RHS of \eqref{eq:Wtensor} is taken over all
$\mathbf{m}=(m^{(a)}_i)_{a\in I, i\geq 1},\, m^{(a)}_i\in {\mathbb Z}_{\geq 0}$
such that $p^{(a)}_i(\mathbf{m}) \geq 0$ for $a\in I, i \geq 1$, and
$$
\sum_{a=1}^n \sum_{i \geq 1} i m^{(a)}_i \alpha_a = 
\sum_{a=1}^n\sum_{i \geq 1} i\nu^{(a)}_i \omega_a - \lambda
\qquad \mbox{ for } a\in I.
$$
Then we have
\begin{equation}\label{eq:fer}
W \cong \bigoplus_{\lambda\in P^{+}}M(W,\lambda)L(\lambda)
\end{equation}
as $U_q(\mathfrak{g})$-modules.

\subsection{Lattice point summation formula}\label{sec:lpsf}
Here we give a review on an alternative formula to \eqref{eq:fer} when $W=\res W_{m}^{(a)}$, conjectured in \cite{MR1745263} and proven in \cite{MR1836791} (see \cite{MR2238884} also). We can rephrase the results in the following way.
\begin{proposition}\label{prop:lpsf}
Let $\theta\in Q$ be the highest root and $\theta=\sum_{a\in I}c_a\alpha_a,\, c_a\in \mathbb{Z}_{>0}$ be its expansion with respect to the basis of simple roots. Assume that $c_a\leq 2$ for $a\in I$. Then there exist a tuple of positive integers $(b_j)_{j\in J_{a}}$ and that of dominant integral weights $(\lambda_j)_{j\in J_{a}}$ for some finite index set $J_{a}$ such that
\begin{equation}\label{eqn:frobsum}
\res W_m^{(a)}=\sum_{\mathbf{x}\in F_m^{(a)}} L(\lambda_{\mathbf{x}})
\end{equation}
where 
$F_m^{(a)}=\{(x_s)_{s\in J_{a}}|\sum_{j\in J_{a}}b_jx_j=m,\, x_j\in \mathbb{Z}_{\geq 0}\}$ and $\lambda_{\mathbf{x}}=\sum_{j\in J_{a}}x_j\lambda_j$ for each $\mathbf{x}\in F_m^{(a)}$. In each case $(b_j)_{j\in J_{a}}$ and $(\lambda_j)_{j\in J_{a}}$ are given as follows :

\textbf{type $A_r$}:
$$
\begin{array}{c|c|c}
 a & (b_j)_{j\in J_{a}} & (\lambda_j)_{j\in J_{a}} \\
\hline
a \text{ arbitrary} & (1) & (\omega_a)
\end{array}
$$

\textbf{type $B_r$}:
$$
\begin{array}{c|c|c}
 a & (b_j)_{j\in J_{a}} & (\lambda_j)_{j\in J_{a}} \\
\hline
\text{$1\leq a\leq r-1$ and even} & (1,1, \dots, 1) & (\omega_0,\omega_2,\dots,\omega_{a})\\ 
\text{$1\leq a\leq r-1$ and odd} & (1,1, \dots, 1) & (\omega_1,\omega_3,\dots , \omega_{a}) \\
\text{$a=r$ and even} & (2,2,\dots, 2,1) & (\omega_0,\omega_2,\dots,\omega_{r-2},\omega_r)\\
\text{$a=r$ and odd} & (2,2, \dots, 2,1) & (\omega_1,\omega_3,\dots,\omega_{r-2},\omega_r)
\end{array}
$$

\textbf{type $C_r$}:
$$
\begin{array}{c|c|c}
 a & (b_j)_{j\in J_{a}} & (\lambda_j)_{j\in J_{a}} \\
\hline
1\leq a\leq r-1 & (2,2,\dots, 2,1) & (\omega_0,\omega_1,\dots \omega_{a-1},\omega_a)\\ 
\text{$a=r$} & (1) & (\omega_a)
\end{array}
$$

\textbf{type $D_r$}:
$$
\begin{array}{c|c|c}
 a & (b_j)_{j\in J_{a}} & (\lambda_j)_{j\in J_{a}} \\
\hline
\text{$1\leq a\leq r-2$ and even}   & (1,1,\dots, 1) & (\omega_0,\omega_2,\dots,\omega_{a})\\ 
\text{$1\leq a\leq r-2$ and odd}   & (1,1,\dots, 1) & (\omega_1,\omega_3,\dots , \omega_{a}) \\
r-1\leq a\leq r & (1) & (\omega_a)\\
\end{array}
$$

\textbf{type $E_6$}:
$$
\begin{array}{c|c|c}
 a & (b_j)_{j\in J_{a}} & (\lambda_j)_{j\in J_{a}} \\
\hline
1 & (1) & (\omega_1) \\
2 & (1,1) & (\omega_2,\omega_5) \\
4 & (1,1) & (\omega_1,\omega_4) \\
5 & (1) & (\omega_5) \\ 
6 & (1,1) & (0,\omega_6) \\
\end{array}
$$

\textbf{type $E_7$}:
$$
\begin{array}{c|c|c}
 a & (b_j)_{j\in J_{a}} & (\lambda_j)_{j\in J_{a}} \\
\hline
1 &(1,1) & (0,\omega_1) \\ 
5 &(1,1,1) & (0,\omega_1,\omega_5) \\
6 &(1) & (\omega_6) \\
7 &(1,1) & (\omega_6, \omega_7)
\end{array}
$$

\textbf{type $E_8$}:
$$
\begin{array}{c|c|c}
 a & (b_j)_{j\in J_{a}} & (\lambda_j)_{j\in J_{a}} \\
\hline
1 & (1,1,1)& (0,\omega_1, \omega_7) \\
7 & (1,1)& (0,\omega_7)
\end{array}
$$

\textbf{type $F_4$}:
$$
\begin{array}{c|c|c}
 a & (b_j)_{j\in J_{a}} & (\lambda_j)_{j\in J_{a}} \\
\hline
1 & (1,1) & (0,\omega_1) \\
4 & (2,2,1) & (0,\omega_1, \omega_4)
\end{array}
$$

\textbf{type $G_2$}:
$$
\begin{array}{c|c|c}
 a & (b_j)_{j\in J_{a}} & (\lambda_j)_{j\in J_{a}} \\
\hline
1 & (1,1) & (0,\omega_1)
\end{array}
$$
\end{proposition}
\begin{definition}
A \textit{rational polyhedron} $P\subset \mathbb{R}^d$ is the set of solutions of a finite system of linear inequalities with integer coefficients :
$$
P=\{\mathbf{x}\in \mathbb{R}^d|\langle \mathbf{b}_i,\mathbf{x} \rangle \leq \beta_i\text{ for }i=1,\dots, l\}
$$
where $\mathbf{b}_i\in \mathbb{Z}^d$, $\beta_i\in \mathbb{Z}$ and $\langle \cdot,\cdot \rangle$ denotes the dot product on $\mathbb{R}^d$.
\end{definition}
We note that the set $F_{m}^{(a)}$ in Proposition \ref{prop:lpsf} can be understood as the set of lattice points in a rational polyhedra as follows. Note that there always exists $j_0\in J_{a}$ such that $b_{j_0}=1$. Assume that $J_{a}'=J_{a}\backslash\{j_0\}$ is non-empty. Then there exists a bijection between $F_m^{(a)}$ and the set $F_m^{(a)'}$ of integer lattice points in the rational polyhedron
$$
P_m^{(a)}=\{(x_s)_{s\in J_{a}'}|\sum_{j\in J_{a}'}a_jx_j\leq m,\, x_j\geq 0 \text{ for } j\in J_{a}'\}.
$$
There are conjectural formulas for $\res W_m^{(a)}$ not covered by proposition \ref{prop:lpsf} in exceptional types \cite{MR1745263}. We will attack some of them in Section \ref{sec:app} as an application of our main theorem. Here we simply remark that they still allow an expression similar to \eqref{eqn:frobsum}. In general, it is of the form
\begin{equation}\label{eq:stepp}
\res W_m^{(a)}=\sum_{\mathbf{x}\in F_m^{(a)'}} p(\mathbf{x})L(\lambda_{\mathbf{x}})
\end{equation}
for some \textit{piecewise step-polynomial} $p$, which means $p:\mathbb{Z}^{|J_{a}'|}\to \mathbb{Q}$ is a collection of polyhedra $Q_i$ in $\mathbb{R}^{|J_{a}'|}$ and step-polynomials $g_i:Q_i\cap \mathbb{Z}^{|J_{a}'|}$ such that $p(\mathbf{x})=g_i(\mathbf{x})$ for $\mathbf{x}\in Q_i\cap \mathbb{Z}^{|J_{a}'|}$. 
A \textit{step-polynomial} is a $\mathbb{Q}$-linear combination of products of functions of the form $\lfloor l(\mathbf{x}) \rfloor$ for some affine function $l:\mathbb{Z}^{|J_{a}'|}\to \mathbb{Q}$ with rational coefficients, i.e., $l(\mathbf{x})=\langle \mathbf{b},\mathbf{x}\rangle+\mathbf{c}$ for some $\mathbf{b},\mathbf{c}\in \mathbb{Q}^{|J_{a}'|}$. We refer the reader to \cite{MR2357077} for more detailed discussion. 

The point is that we can regard $\{P_m^{(a)}\}_{m\in \mathbb{Z}_{\geq 0}}$ as a 1-parameter family of rational polyhedra whose ``combinatorial structures" do not change as $m$ varies. This uniformity is what allows us to treat various summations over $F_m^{(a)'}$ in an $m$-independent way; see Theorem 5.3 in \cite{MR1731815} and Proposition 2.6 in \cite{MR2357077} for example. It also leads us to expect the sequence of the characters of the RHS in \eqref{eq:stepp} to satisfy a linear recurrence relation. Although we do not directly use any general results from \cite{MR1731815} or \cite{MR2357077} in this paper, they are certainly useful in getting some new perspectives on the topic.

\section{Proof of the main theorem}\label{sec:proof}
In this section we prove Theorem \ref{thm:main}. In Section \ref{sec:lattice} we consider the cases where there exist lattice point summation formulas available. Then in Section \ref{sec:exceptional} we consider the remaining cases. 

Let us begin with some common properties of $\Lambda_a$ and $\Lambda'_a$ given in Appendix \ref{sec:lambda}.
\begin{proposition}\label{prop:lacommon}
Let $\mathfrak{g}$ be a simple Lie algebra and $a\in I$ be listed in Theorem \ref{thm:main}. Let $\Lambda_a$ and $\Lambda'_a$ be as given in Appendix \ref{sec:lambda}. They satisfy the properties \ref{thm:main:1}, \ref{thm:main:3}, \ref{thm:main:4}, \ref{thm:main:6} and \ref{thm:main:7} in Theorem \ref{thm:main}. 
\end{proposition}
\begin{proof}
\ref{thm:main:1} is true as $\Lambda_a$ and $\Lambda'_a$ are disjoint unions of $W$-orbits by definition.

\ref{thm:main:3} follows from the facts that $\mathcal{O}(t\lambda)=\{t \mu|\mu\in \mathcal{O}(\lambda)\}$ for $t\in \mathbb{Z}_{\geq 0},\, \lambda\in P^{+}$ and that $\mathcal{O}(\lambda)\cap \mathcal{O}(\mu)=\emptyset$ for distinct $\lambda,\mu\in P^{+}$.

\ref{thm:main:4} follows from a direct inspection of $\Lambda_a$. 

To prove \ref{thm:main:6}, we recall the fact that if $\mu$ is a weight of $L(\lambda)$, then $\lambda-\mu\geq 0$. We can check $\Lambda_a$ is a subset of $\Omega\left(L(\omega_a)\right)$ from the given description of 
$\Omega\left(L(\omega_a)\right)$ in Appendix \ref{sec:chwo} and $\Lambda_a$ in Appendix \ref{sec:lambda}. This implies \ref{thm:main:6}.

Finally, we need to prove \ref{thm:main:7}. We note that $\omega_a\geq 0$ for each $a\in I$ as the inverse of the Cartan matrix always has positive rational entries. Now we have to check the following:

In type $B_r$, $2\omega_r-\lambda\geq 0$ for $\lambda\in \Lambda'_r$; it follows from $2\omega_r-\omega_{r-2}=\alpha_{r-1}+2\alpha_{r}\geq 0$ and $\omega_{r-2}-\lambda\geq 0$ for $\lambda \in \Lambda'_r=\Lambda_{r-2}$, which is true from \ref{thm:main:6}.

In type $C_r$, $2\omega_a-\lambda\geq 0$ for $\lambda\in \{\omega_0,\omega_1,\dots,\omega_{a-1}\}$ for each $a\in I\backslash{\{r\}}$; we have $\omega_a\geq 0$ for each $a\in I$ and 
$$\omega_j-\omega_{j-1}=\frac{\alpha_r}{2}+\sum_{i=j}^{r-1}\alpha_i \geq 0, \forall j\in I.$$

In type $F_4$, $2\omega_3-\lambda\geq 0$ for $\lambda \in \Lambda'_3\cap P^+=\{0,\omega _1,\omega _2,2 \omega _1,2 \omega _4,\omega _1+2 \omega _4\}$ and $2\omega_4-\lambda\geq 0$ for $\lambda \in \{0,\omega_1\}=\Lambda'_4\cap P^+=\Lambda_1\cap P^+$; as $\Lambda_1\cap P^+=\{0,\omega _1\}$ and $\Lambda_2\cap P^+=\{0,\omega _1,\omega _2, 2 \omega _4\}$, for which we already know $\omega_a-\lambda\geq 0,\forall \lambda \in  \Lambda_a$ when $a\in \{1,2\}$ from \ref{thm:main:6}, it is enough to check
$$
\begin{aligned}
&2 \omega _3-\omega _2&=\,&\alpha _1+2 \alpha _2+4 \alpha _3+2 \alpha _4 &\geq 0, \\
&2 \omega _3-2 \omega _1&=\,&2 \alpha _2+4 \alpha _3+2 \alpha _4 &\geq 0, \\
&2 \omega _3-\omega _1-2 \omega _4&=\,&\alpha _2+2 \alpha _3 &\geq 0,\\
&2 \omega _4-\omega _1&=\,&\alpha _2+2 \alpha _3+2 \alpha _4 &\geq 0. \\
\end{aligned}
$$

In type $G_2$, $3\omega_2-\lambda\geq 0$ where $\lambda \in \{0,\omega_1\}=\Lambda_1\cap P^{+}$; it is enough to check
$
 3 \omega _2-\omega _1=\alpha _1+3 \alpha _2\geq 0
$.
\end{proof}
Now it is enough to prove Theorem \ref{thm:main} \ref{thm:main:2} when  $\Lambda_a$ and $\Lambda'_a$ are given as in Appendix \ref{sec:lambda}. Thus from now on we will focus on showing this for the rest of the section.
\subsection{Cases listed in Proposition \ref{prop:lpsf}}\label{sec:lattice}
\begin{proposition}\label{prop:lpsfL}
Assume that a tuple of positive integers $(b_j)_{j\in J}$ and that of dominant integral weights $(\lambda_j)_{j\in J}$ for some finite index set $J$ are given. Let 
$F_m=\{(x_s)_{s\in J}|\sum_{j\in J}b_jx_j=m,\, x_j\in \mathbb{Z}_{\geq 0}\}$ and $\lambda_{\mathbf{x}}=\sum_{j\in J}x_j\lambda_j$ for each $\mathbf{x}\in F_m$. Then the sequence $\{Q_m\}_{m=0}^{\infty}$ of characters defined by
$$
Q_m=\sum_{\mathbf{x}\in F_m} \chi\left(L(\lambda_{\mathbf{x}})\right)
$$
satisfies 
$
\mathcal{L}[Q_{m}]=0
$
where
$$
\mathcal{L}=\prod_{j\in J}\prod_{\lambda\in \mathcal{O}(\lambda_j)}(1-e^{\lambda} \Delta^{b_j}).
$$
\end{proposition}
\begin{proof}
From the Weyl character formula, we have
$$
\left(\sum_{w\in W} (-1)^{\ell(w)}w(e^{\rho})\right)\chi \left(L(\lambda_{\mathbf{x}})\right)=\sum_{w\in W} (-1)^{\ell(w)}w(e^{\rho})w(e^{\sum_{j\in J}x_j\lambda_j})
$$
and thus
$$
\left(\sum_{w\in W} (-1)^{\ell(w)}w(e^{\rho})\right)Q_m=\sum_{w\in W} (-1)^{\ell(w)}w(e^{\rho}) P_{m,w} $$
where
$$
P_{m,w}=\sum_{\mathbf{x}\in F_m} e^{\sum_{j\in J}x_jw(\lambda_j)}
$$
for each $m\in \mathbb{Z}_{\geq 0}$ and $w\in W$.
One can easily check
$$
\sum_{m=0}^{\infty}P_{m,w}t^m=\prod_{j\in J}\frac{1}{1-e^{w(\lambda_j)} t^{b_j}}.
$$
Thus
$$
\left(\sum_{w\in W} (-1)^{\ell(w)}w(e^{\rho})\right)\sum_{m=0}^{\infty}Q_{m}t^m =\sum_{w\in W} (-1)^{\ell(w)}w(e^{\rho})\prod_{j\in J}\frac{1}{1-e^{w(\lambda_j)} t^{b_j}}
$$
and Proposition \ref{prop:genL} gives the desired conclusion.
\end{proof}
Now we are ready to Theorem \ref{thm:main} for the cases covered in Proposition \ref{prop:lpsf}.
\begin{proof}
Now assume that $a\in I$ is listed in Proposition \ref{prop:lpsf}. Let
$\Lambda_a$ and $\Lambda'_a$ be as given in Appendix \ref{sec:lambda}. From Proposition \ref{prop:lpsf} and \ref{prop:lpsfL}, we have
$
\mathcal{L}_a[Q_{m}^{(a)}]=0
$
where
$$
\mathcal{L}_a=\prod_{\lambda \in \Lambda_a}(1-e^{\lambda}\Delta)\prod_{\lambda \in \Lambda'_a}(1-e^{\lambda}\Delta^{t_a}).
$$
This proves \ref{thm:main:2}.
\end{proof}

\subsection{The remaining cases}\label{sec:exceptional}
\begin{lemma}\label{lem:der2}
Assume that the pair $(a,b)$ of elements in $I$ satisfies the following conditions : (i) both $\alpha_a$ and $\alpha_b$ are long roots, (ii) $a$ is the unique node adjacent to $b$ and (iii) there exists a $W$-invariant set $\Lambda_b$ such that $\mathcal{L}_b[Q_{m}^{(a)}]=0$ where 
$$\mathcal{L}_b:=\prod_{\lambda \in \Lambda_b}(1-e^{\lambda}\Delta).$$
Then $\Lambda_a:=\{\lambda+\mu|\lambda,\mu \in \Lambda_b, \lambda\neq \mu\}$ is $W$-invariant and $\mathcal{L}_a[Q_{m}^{(a)}]=0$ where $$\mathcal{L}_a:=\prod_{\lambda \in \Lambda_a}(1-e^{\lambda}\Delta).$$ Furthermore, if there exists a $U_q(\mathfrak{g})$-module $V_b$ such that $\Omega(V_b)=\Lambda_b$, then $\Lambda_a=\Omega(\extp^2 V_b)$.
\end{lemma}
\begin{proof}
As $\Lambda_b$ is $W$-invariant, $\Lambda_a$ is also $W$-invariant. From \eqref{eq:Qsys}, we have 
$$\left(Q_m^{(b)}\right)^2=Q_{m+1}^{(b)}Q_{m-1}^{(b)}-Q_m^{(a)},$$
or
$$Q_m^{(a)}=\left(Q_m^{(b)}\right)^2-Q_{m+1}^{(b)}Q_{m-1}^{(b)}.$$
Then we can apply Lemma \ref{lem:der} to conclude
$
\mathcal{L}_a[Q_{m}^{(a)}]=0.
$

The last statement follows from the fact that a weight of $\extp^2 V_b$ is always of the form $\lambda+\mu$ for some $\lambda,\mu \in \Omega(V_b)=\Lambda_b,$ $\lambda\neq \mu$.
\end{proof}

The next two statements are useful when we compute the set of weights in the second exterior power of a fundamental representation of $U_q(\mathfrak{g})$.
\begin{lemma}\label{lem:der3}
Let $\mathfrak{g}$ be a simple Lie algebra. For each $a\in I$, we have
$$
\Omega\left(\extp^2 L(\omega_a)\right)=\Omega\left(L(2\omega_a-\alpha_a)\right).
$$
\end{lemma}
\begin{proof}
Let us first show that any weight $\lambda \in L(\omega_a)$, different from $\omega_a$ satisfies $\omega_a -\lambda \succeq \alpha_a$. If $\lambda\neq \omega_a$, then one can always find $b\in I$ such that $\lambda+\alpha_b\in \Omega(L(\omega_a))$; otherwise a non-zero weight vector $v\in L(\omega_a)$ of weight $\lambda$ is annihilated by each Chevalley generator $E_i\in U_{q}(\mathfrak{g}),\, i\in I$ and then $v$ generates a proper $U_q(\mathfrak{g})$-submodule of the irreducible module $L(\omega_a)$, which is a contradiction. Therefore $\lambda$ not equal to $\omega_a$ satisfies $\omega_a-\alpha_b \succeq \lambda$ for some $b\in I$ with $\omega_a-\alpha_b\in L(\omega_a)$ by adding simple roots to $\lambda$ repeatedly. But now we claim that $\omega_a-\alpha_b\in \Omega(L(\omega_a)),\, b\in I$ only when $b=a$. When $b=a$, $\omega_a-\alpha_a = s_a(\omega_a)$ must be clearly in $\Omega(L(\omega_a))$ as $\Omega(L(\omega_a))$ is $W$-invariant. If $b\neq a$, then $s_b(\omega_a-\alpha_b) = \omega_a +\alpha_b$, which cannot be in $\Omega(L(\omega_a))$ and then again by $W$-invariance, $\omega_a-\alpha_b\notin \Omega(L(\omega_a))$.

Now let $\nu_a := 2\omega_a-\alpha_a$. Note that $\nu_a$ is a dominant integral weight. From what we just saw in the above paragraph, we can conclude that for any $\lambda,\mu \in \Omega(L(\omega_a))$ with $\lambda\neq \mu$, we have $\nu_a \succeq \lambda+\mu$. This implies $\extp^2 L(\omega_a)$ contains an isomorphic copy of $L(\nu_a)$, which implies $\Omega\left(L(\nu_a)\right) \subseteq \Omega\left(\extp^2 L(\omega_a)\right)$. 

To prove the reverse inclusion $\Omega\left(\extp^2 L(\omega_a)\right)\subseteq \Omega\left(L(\nu_a)\right)$, it suffices to show that any dominant weight in $\Omega\left(\extp^2 L(\omega_a)\right)$ is also in $\Omega\left(L(\nu_a)\right)$ as both of them are $W$-invariant. Let $\lambda+\mu\in \Omega\left(\extp^2 L(\omega_a)\right)$ be a dominant weight, where $\lambda,\mu \in \Omega(L(\omega_a)),\, \lambda\neq \mu$. Then at least one of them, say, $\mu$ must satisfy $\omega_a-\alpha_a\succeq \mu$. Therefore $2\omega-\alpha_a \succeq \lambda+\mu$, which implies that $\lambda+\mu$ is in $\Omega\left(L(\nu_a)\right)$. This proves the lemma.
\end{proof}

\begin{prop}\label{prop:der4}
Let $\mathfrak{g}$ be a simple Lie algebra. 
For each $a\in I$, we have
$$
\Omega\left(\extp^2 L(\omega_a)\right)=\Omega\left(\bigotimes_{b : C_{ab}\neq 0} L(\omega_b)^{\otimes (-C_{ba})} \right).
$$
\end{prop}
\begin{proof}
By Lemma \ref{lem:der3}, we have
$$
\Omega\left(\extp^2 L(\omega_a)\right)=\Omega\left(L(2\omega_a-\alpha_a)\right).
$$
Recall the following relation between the fundamental weights and simple roots 
$\alpha_a=\sum_{b\in I}C_{ba}\omega_b$, which implies
$$2\omega_a-\alpha_a=-\sum_{b:C_{ab}\neq 0 }C_{ba}\omega_b.$$

Thus now it is enough to show that the set of weights of two $U_q(\mathfrak{g})$-modules $\bigotimes_{b : C_{ab}\neq 0} L(\omega_b)^{\otimes (-C_{ba})}$ and $L(-\sum_{b : C_{ab}\neq 0}C_{ba}\omega_b)$ are the same. It can be done by applying the following fact repeatedly : for two dominant weights $\lambda$ and $\mu$, we have
$$
\Omega\left(L(\lambda)\otimes L(\mu)\right) = \Omega\left(L(\lambda+\mu)\right).
$$
This is true because $L(\lambda)\otimes L(\mu)$ contains an isomorphic copy of $L(\lambda+\mu)$ as a $U_q(\mathfrak{g})$-submodule and any dominant weight $\lambda'+\mu'\in \Omega\left(L(\lambda)\otimes L(\mu)\right)$ with $\lambda' \in \Omega\left(L(\lambda)\right)$ and $\mu'\in \Omega\left(L(\mu)\right)$ is also a weight of $L(\lambda+\mu)$, which implies the reciprocal inclusion between the two sets of weights.
\end{proof}

For the rest of this section, we consider each exceptional type separately to complete the proof of Theorem \ref{thm:main} \ref{thm:main:2}.

\subsubsection{Type $E_6, E_7, E_8$}
Here we assume that $\mathfrak{g}$ is of type $E_r,\, r\in\{6,7,8\}$.
\begin{proof}
Let $\Lambda_a=\Omega \left(L(\omega_a)\right)$ be as given in Appendix \ref{sec:lambda} for each $a\in I$. In the previous section, we have completed our proof of Theorem \ref{thm:main} for  
$$
a \in  
\begin{cases}
\{1,2,4,5,6\}, & \text{when $\mathfrak{g}$ is of type $E_6$},\\
\{1,5,6,7\}, & \text{when $\mathfrak{g}$ is of type $E_7$},\\
\{1,7\}, & \text{when $\mathfrak{g}$ is of type $E_8$}.
\end{cases}
$$

Now consider the remaining cases. Let $(a,b)$ be a pair of elements of $I$ satisfying the assumption of Lemma \ref{lem:der2}, say,
$$
(a,b) = 
\begin{cases}
(3,6), & \text{when $\mathfrak{g}$ is of type $E_6$},\\
(2,1) \text{ or } (3,7), & \text{when $\mathfrak{g}$ is of type $E_7$},\\
(2,1) \text{ or } (6,7), & \text{when $\mathfrak{g}$ is of type $E_8$}.
\end{cases}
$$
In this case, we have
$$\Lambda_a=\Omega\left(L(\omega_a)\right)=\Omega\left(\extp^2 L(\omega_b)\right)$$
from Proposition \ref{prop:der4}. Then we can apply Lemma \ref{lem:der2} to complete the proof.
\end{proof}
\subsubsection{Type $F_4$}
Assume that $\mathfrak{g}$ is of type $F_4$. 
Theorem \ref{thm:main} is proved for $a\in \{1,4\}$ in Section \ref{sec:lattice}. Now we prove it for $a\in \{2,3\}$. As $\alpha_2$ is a long root and $\alpha_3$ is a short root, we treat the two cases separately.
\begin{proof}
Let $\Lambda_a$ and $\Lambda'_a$ be as given in Appendix \ref{sec:lambda} for each $a\in I$. We can apply Lemma \ref{lem:der2} for the pair $(a,b)=(2,1)$. Recall that 
$\Lambda_1=\mathcal{O}(0)\coprod \mathcal{O}(\omega_1)$. As $\Lambda_1$ is a proper subset of $\Omega\left(L(\omega_1)\right)$, Proposition \ref{prop:der4} does not apply and thus we find $\{\lambda+\mu|\lambda,\mu \in \Lambda_1, \lambda\neq \mu\}$ by direct computation. The elements of $\Lambda_1$ are 
$$
\begin{array}{llll}
 \omega _1 , & \omega _2-2 \omega _3 , & \omega _1+\omega _2-2 \omega _3 , & -\omega _2+2 \omega _3-2 \omega _4 \\
 -\omega _1 , & 2 \omega _3-\omega _2 , & -\omega _1+2 \omega _2-2 \omega _3 , & \omega _1-\omega _2+2 \omega _4 \\
 \omega _1-\omega _2 , & \omega _1-2 \omega _4 , & \omega _1-2 \omega _2+2 \omega _3 , & \omega _1-2 \omega _3+2 \omega _4 \\
 2 \omega _1-\omega _2 , & \omega _2-2 \omega _4 , & -\omega _1-\omega _2+2 \omega _3 , & \omega _2-2 \omega _3+2 \omega _4 \\
 \omega _2-2 \omega _1 , & 2 \omega _4-\omega _1 , & -\omega _1+\omega _2-2 \omega _4 , & \omega _1-\omega _2+2 \omega _3-2 \omega _4 \\
 \omega _2-\omega _1 , & 2 \omega _4-\omega _2 , & -\omega _1+2 \omega _3-2 \omega _4 , & -\omega _1+\omega _2-2 \omega _3+2 \omega _4, \\
\end{array}
$$
and $0$.
We can check
$$
\{\lambda+\mu|\lambda,\mu \in \Lambda_1, \lambda\neq \mu\}=\mathcal{O}(0)\coprod \mathcal{O}(\omega_1)\coprod\mathcal{O}(\omega_2)\coprod\mathcal{O}(2\omega_4)
$$
in a straightforward way as it is a simple matter of finding dominant integral weights in the set. As this is the same as $\Lambda_2$, Lemma \ref{lem:der2} implies \ref{thm:main:2} for $a=2$.

Now consider the case of $a=3$. For the pair $(a,b)=(4,3)$, we are not able to apply Lemma \ref{lem:der2} directly as $\alpha_3$ and $\alpha_4$ are both short roots. But we can still proceed using Lemma \ref{lem:der} since $Q_m^{(3)}=\left(Q_m^{(4)}\right)^2-Q_{m+1}^{(4)}Q_{m-1}^{(4)}$ in \eqref{eq:Qsys}
still holds.

Recall that $\Lambda_4=\mathcal{O}(\omega_4)$ and $\Lambda_4'=\mathcal{O}(0)\coprod\mathcal{O}(\omega_1)$. We already have shown that
$
\mathcal{L}_4[Q_{m}^{(4)}]=0
$
where
$$
\mathcal{L}_4=\prod_{\lambda \in \Lambda_4}(1-e^{\lambda}\Delta)\prod_{\lambda \in \Lambda'_4}(1-e^{\lambda}\Delta^{2}).
$$

From Lemma \ref{lem:der} we know that 
$
\mathcal{L}_3'[Q_{m}^{(3)}]=0
$
where $\mathcal{L}_3'$ is a difference operator without multiple roots and whose set of roots consists of all elements of the following forms :
$$
\begin{array}{llll}
(1)& e^{\lambda+\mu}, &\lambda,\mu \in \Lambda_4,&\, \lambda\neq \mu \\
(2)& e^{\lambda/2+\mu/2},& \lambda, \mu\in \Lambda_4',&\, \lambda\neq \mu \\
(3)& e^{\lambda+\mu/2},& \lambda\in \Lambda_4, \mu\in \Lambda_4', \\
(3')& -e^{\lambda+\mu/2},& \lambda\in \Lambda_4, \mu\in \Lambda_4', \\
(4)& -e^{\lambda/2+\mu/2},& \lambda\in \Lambda_4', \mu\in \Lambda_4'.
\end{array}
$$

Now we check the following in a straightforward way as before : 
$$
\begin{aligned}
S_1&=\{\lambda+\mu|\lambda,\mu \in \Lambda_4,\, \lambda\neq \mu\}=\mathcal{O}(\omega_1)\coprod \mathcal{O}(\omega_3)\coprod \mathcal{O}(\omega_4)\coprod \mathcal{O}(0), \\
S_2&=\{\lambda/2+\mu/2|\lambda, \mu\in \Lambda_4'\, \lambda\neq \mu\}=\mathcal{O}\left(\omega_1/2\right)\coprod \mathcal{O}\left(\omega_2/2\right)\coprod \mathcal{O}\left(\omega_4\right)\coprod \mathcal{O}(0), \\
S_3&=\{\lambda+\mu/2|\lambda\in \Lambda_4, \mu\in \Lambda_4'\}=\mathcal{O}\left(\omega_1/2\right)\coprod \mathcal{O}\left(\omega_2/2\right)\coprod \mathcal{O}\left(\omega_4\right)\coprod \mathcal{O}\left(\omega_1/2+\omega _4\right),\\
S_4&=\{\lambda/2+\mu/2|\lambda\in \Lambda_4', \mu\in \Lambda_4'\}=S_2\coprod \mathcal{O}\left(\omega _1\right)
\end{aligned}
$$
Therefore the set of roots of $\mathcal{L}_3'$ is the union of the sets
$$
\mathcal{O}_{e}(\omega_3), \pm \mathcal{O}_{e}(\omega_1),  \pm \mathcal{O}_{e}(\omega_4),\pm \mathcal{O}_{e}(0), \pm \mathcal{O}_{e}(\omega_1/2),\pm \mathcal{O}_{e}(\omega_2/2), \pm \mathcal{O}_{e}(\omega_1/2+\omega_4)
$$
where $-\mathcal{O}_{e}(\lambda), \lambda\in \mathfrak{h}_{\mathbb{R}}^{*}$ denotes the set $\{-x|x\in \mathcal{O}_{e}(\lambda)\}$.

Thus we find
$$
\mathcal{L}_3'=\mathcal{L}_3=\prod_{\lambda \in \Lambda_3}(1-e^{\lambda}\Delta)\prod_{\lambda \in \Lambda'_3}(1-e^{\lambda}\Delta^{2})
$$
which proves \ref{thm:main:2} for $a=3$. 
\end{proof}
\subsubsection{Type $G_2$}
Assume that $\mathfrak{g}$ is of type $G_2$. Theorem \ref{thm:main} holds for $a=1$ as shown in Section \ref{sec:lattice}. To treat the case when $a=2$, we employ a different strategy from the previous section because Lemma \ref{lem:der2} does not apply directly. 

Let $\Lambda_a$ and $\Lambda'_a$ as given in Appendix \ref{sec:lambda} for each $a\in I$. Let us consider a pair of sequences $\{R_m^{(1)}\}_{m=0}^{\infty}$ and $\{R_m^{(2)}\}_{m=0}^{\infty}$ given by
\begin{equation}\label{eq:gr1}
R_m^{(1)}=\sum_{\lambda\in \Lambda_1}c_{\lambda}e^{m \lambda}
\end{equation}
and
\begin{equation}\label{eq:gr2}
\begin{aligned}
R_{3m}^{(2)}&=\sum_{\lambda\in \Lambda_2}c_{\lambda}e^{3m \lambda}&+\sum_{\lambda\in \Lambda_2'}t_{\lambda,0} e^{m\lambda}, \\
R_{3m+1}^{(2)}&=\sum_{\lambda\in \Lambda_2}c_{\lambda}e^{\lambda}e^{3m \lambda}&+\sum_{\lambda\in \Lambda_2'}t_{\lambda,1} e^{m\lambda}, \\
R_{3m+2}^{(2)}&=\sum_{\lambda\in \Lambda_2}c_{\lambda}e^{2\lambda}e^{3m \lambda}&+\sum_{\lambda\in \Lambda_2'}t_{\lambda,2} e^{m\lambda}
\end{aligned}
\end{equation}
for some $c_{\lambda}\in \mathbb{C}(y_1,y_2),\, \lambda\in \Lambda_1\coprod \Lambda_2$ and $t_{\lambda,j}\in \mathbb{C}(y_1,y_2),\, \lambda\in \Lambda'_2,j\in \{0,1,2\}$. Let $c_{\lambda}$ and $t_{\lambda,j}$ be as given in Appendix \ref{sec:appg2}.

\begin{lemma}\label{lem:g2}
The family $\left(R_m^{(a)}\right)_{a\in I,m\in \mathbb{Z}_{\geq 0}}$ satisfies the $Q$-system of type $G_2$.
\end{lemma}
\begin{proof}
We can check the following, which is an explicit version of \eqref{eq:Qsys}, in a straightforward way using computer algebra systems:
\begin{equation}\label{eq:g2Q}
\begin{aligned}
\left(R^{(1)}_m\right)^2
  &= R_{m-1}^{(1)} R_{m+1}^{(1)} + R^{(2)}_{3m}, \\
\left(R^{(2)}_{3m}\right)^2
  &= R_{3m-1}^{(2)} R_{3m+1}^{(2)}
+ \left(R^{(1)}_m\right)^3, \\
\left(R^{(2)}_{3m+1}\right)^2
  &= R_{3m}^{(2)} R_{3m+2}^{(2)}
+ \left(R^{(1)}_m\right)^2 R^{(1)}_{m+1}, \\
\left(R^{(2)}_{3m+2}\right)^2
  &= R_{3m+1}^{(2)} R_{3m+3}^{(2)}
+ R^{(1)}_{m} \left(R^{(1)}_{m+1}\right)^2.
\end{aligned}
\end{equation}
Note that any product of various $R_m^{(a)}$'s in \eqref{eq:g2Q} can be written in the form of $\sum_{\alpha} d_{\alpha}\alpha^{m}$ for some $d_{\alpha}$ and $\alpha$ in $\mathbb{C}(y_1,y_2)$. Once we write each of them in this form, verifying \eqref{eq:g2Q} is reduced to checking $m$-independent relations among the coefficients for each exponential term $\alpha^m$, which thus requires a finite amount of computation. As the sizes of $\Lambda_a$ and $\Lambda_a'$ are just 6 or 7, the number of relations to be checked is reasonably small.
\end{proof}

\begin{proposition}\label{prop:QR}
We have $Q_m^{(a)}=R_m^{(a)}$ for each $a\in I$ and $m\in \mathbb{Z}_{\geq 0}$. 
\end{proposition}

\begin{proof}
As both $\{Q_m^{(a)}\}_{a\in I}$ and $\{R_m^{(a)}\}_{a\in I}$ are solutions of the $Q$-system of type $G_2$ and $Q_m^{(a)}$ is non-zero for any $(a,m)$, it is enough to check the initial conditions of $Q$-systems, $Q_1^{(1)}=R_1^{(1)}$ and $Q_1^{(2)}=R_1^{(2)}$, which can be done in a straightforward way.
\end{proof}

Now we finish the proof of Theorem \ref{thm:main}.
\begin{proof}
Let $\Lambda_a$ and $\Lambda_a'$ be as given in Appendix \ref{sec:lambda} for each $a\in I$. Note that we can rewrite \eqref{eq:gr2} as
$$
R_m^{(2)}=\sum_{\lambda\in \Lambda_2}c_{\lambda}e^{m \lambda}+\sum_{
\substack{
\lambda\in \Lambda_2' \\
j\in \{0,1,2\}}}
d_{\lambda,j}(\zeta^j e^{\lambda/3})^m
$$
where $\zeta=e^{2\pi i/3}$ and $(d_{\lambda,j})$ is the unique solution of the system of equations
$$
t_{\lambda,k}=\sum_{j=0}^{2}d_{\lambda,j}(\zeta^j e^{\lambda/3})^k,\, \lambda\in \Lambda_2', k\in\{0,1,2\}.
$$
Thus from proposition \ref{prop:cr} we have
$
\mathcal{L}_2[R_{m}^{(2)}]=0
$
where $$
\mathcal{L}_2=\prod_{\lambda \in \Lambda_2}(1-e^{\lambda}\Delta)\prod_{\lambda \in \Lambda'_2}(1-e^{\lambda}\Delta^{3}).
$$
Proposition \ref{prop:QR} implies
$$
\mathcal{L}_2[Q_{m}^{(2)}]=0,
$$
which proves \ref{thm:main:2} for $a=2$. 
\end{proof}

\subsection{Some remarks}
When $\mathfrak{g}$ is simply-laced, the structure of the set of roots of the difference operators in Theorem \ref{thm:main} and the $Q$-system \eqref{eq:Qsys} suggest the following balancing condition on roots, considering the exponential behavior of characters : 
$$
\Omega\left(\extp^2 L(\omega_a)\right)=\Omega\left(\bigotimes_{b : C_{ab}\neq 0} L(\omega_b)\right),\quad a\in I
$$
which is a special case of Proposition \ref{prop:der4}. It was actually conceived in this way of reasoning. It is quite amusing to see that $Q$-systems tell us something new about the fundamental representations $L(\omega_a)$ of $U_q(\mathfrak{g})$.

The method used in the proofs of Lemma \ref{lem:g2} and proposition \ref{prop:QR} can be applied in principle whenever one can guess $\Lambda_a$ and $\Lambda_a'$. Thus even in type $E_7$ or $E_8$ where we have some nodes for which we do not yet have a proof of Theorem \ref{thm:main}, there is still possibility to carry out the same program as it is a finite problem in linear algebra. The difficult part is to come up with the coefficients $c_{\lambda}$ such that $Q_m^{(a)}=\sum_{\lambda\in \Lambda_a} c_{\lambda}e^{m\lambda}$ for $m=0,\dots, |\Lambda_a|-1$ due to the huge size of $\Lambda_a$ and complicated expressions of $Q_m^{(a)}$ even for relatively small $m$. We will leave it for future work.

\section{Application (I) : decomposition of KR modules in exceptional types}\label{sec:app}
In Section \ref{sec:lpsf}, we have seen some examples where the decomposition of $\res \, {W^{(a)}_m}$ is given by a lattice point summation formula \eqref{eqn:frobsum}. In this section, we apply Theorem \ref{thm:main} to attack the following :
\begin{conjecture}\cite{MR1745263}\label{conj:sumef}
Let $\mathfrak{g}$ be a simple Lie algebra listed below. For any $m\in \mathbb{Z}_{\geq 0}$, we have
\begin{equation}\label{eqn:sumef}
\res \, {W^{(a)}_m} = 
\sum_{
	\substack{
		j_1 + 2 j_2 + j_3 + j_4 \le m\\
		j_1,j_2,j_3,j_4 \in \mathbb{Z}_{\geq 0}
	}}
    p(j_1,j_2,j_3,j_4) L\left( j_1\lambda_1 +j_2\lambda_2 +j_3 \lambda_3 + j_4 \lambda_4 \right),
\end{equation}
where 
$$
p(j_1,j_2,j_3,j_4)=\min \left( 1+j_3,\,1+m-j_1-2 j_2-j_3-j_4 \right) (j_4+1)
$$
and
\begin{equation}\label{eqn:sumefla}
(\lambda_1,\lambda_2,\lambda_3,\lambda_4)=
\begin{cases}
(\omega_1+\omega_5,\omega_2+\omega_4, \omega_3 ,\omega_6), & \text{if $\mathfrak{g}$ is of type $E_6$ and $a=3$}\\
(\omega_5,\omega_3,\omega_2,\omega_1), & \text{if $\mathfrak{g}$ is of type $E_7$ and $a=2$}\\
(\omega_1,\omega_5,\omega_6,\omega_7), & \text{if $\mathfrak{g}$ is of type $E_8$ and $a=6$}\\
(2\omega_4,2\omega_3,\omega_2,\omega_1), & \text{if $\mathfrak{g}$ is of type $F_4$ and $a=2$.}\\
\end{cases}
\end{equation}
\end{conjecture}
Using the same principle, we also give a new proof of the following :
\begin{theorem}\cite{MR2372556}\label{thm:sumg2}
Let $\mathfrak{g}$ be a simple Lie algebra of type $G_2$. We have
\begin{equation}\label{eqn:sumg2}
\begin{aligned}
\res \ W^{(2)}_m=\sum_{\substack{
		j_0+3j_1 + j_2 = m\\
		j_0,j_1,j_2 \in \mathbb{Z}_{\geq 0}
        }}
p(j_0,j_1,j_2) L(j_1\omega_1+j_2\omega_2)
\end{aligned}
\end{equation}
where
$$p(j_0,j_1,j_2)=\left\{
\min\left(
	\Bigl \lfloor \frac{j_0}{3} \Bigr \rfloor,\,
	\Bigl \lfloor \frac{j_2-j_0}{3} \Bigr \rfloor + \Bigl \lfloor \frac{j_0}{3} \Bigr \rfloor
	\right)
+1 \right\}( j_1+1 ).$$
\end{theorem}
\begin{remark}
Theorem \ref{thm:sumg2} is originally conjectured in the following form in \cite{MR1745263}:
$$
\begin{aligned}
\res \ W^{(2)}_m = &
\sum_{k=0}^{\lfloor m/3\rfloor } \sum_{x=2k}^{m-k}
\left\{
\min\left(
	\Bigl \lfloor  \frac{x-2k}{3} \Bigr \rfloor ,\,
	\Bigl \lfloor  \frac{m+k-2x}{3} \Bigr \rfloor  + \Bigl \lfloor  \frac{x-2k}{3} \Bigr \rfloor 
	\right)
+1 \right\}\\
&\quad\quad\quad\quad \times (k+1)L(k \omega_1+(m-x-k)\omega_2).
\end{aligned}
$$
We can easily show that this is the same as \eqref{eqn:sumg2}. See also \cite{MR3438313} for a combinatorial study regarding this formula.
\end{remark}
Let $P_m^{(a)}$ be the character of the RHS of \eqref{eqn:sumef} or \eqref{eqn:sumg2}, for which we want to show $Q_m^{(a)}=P_m^{(a)}$ for all $m\in \mathbb{Z}_{\geq 0}$. As we will see soon, we can easily find a linear recurrence relation that $\{P_m^{(a)}\}_{m=0}^{\infty}$ satisfies. Then as we have found a linear recurrence relation for $\{Q_m^{(a)}\}_{m=0}^{\infty}$, it is enough to construct a single linear recurrence relation satisfied by both $\{P_m^{(a)}\}_{m=0}^{\infty}$ and $\{Q_m^{(a)}\}_{m=0}^{\infty}$ and to check the equality $Q_m^{(a)}=P_m^{(a)}$ for a finite number of $m$'s to conclude $Q_m^{(a)}=P_m^{(a)}$ for all $m\in \mathbb{Z}_{\geq 0}$.
\subsection{Reduction of Conjecture \ref{conj:sumef} to a finite problem}
\begin{lemma}\label{lem:sumef}
Let $\{P_m\}_{m=0}^{\infty}$ be the sequence defined by
$$
P_m=\sum_{
	\substack{
		j_1 + 2 j_2 + j_3 + j_4 \le m\\
		j_1,j_2,j_3,j_4 \in \mathbb{Z}_{\geq 0}
	}
	}
\min \left( 1+j_3,\,1+m-j_1-2 j_2-j_3-j_4 \right)\,(j_4+1)\,x_1^{j_1}x_2^{j_2}x_3^{j_3}x_4^{j_4}.
$$
We have
$$
\sum_{m=0}^{\infty}P_mt^m=\frac{1}{(1-t) \left(1-t x_1\right) \left(1-t^2 x_2\right) \left(1-t x_3\right) \left(1-t^2 x_3\right) \left(1-t x_4\right){}^2}.
$$
\end{lemma}
\begin{proof}
We have $P_m=a_m+b_m$ where
\begin{equation}\label{eq:ef1}
a_m=\sum_{
	\substack{
		j_0+j_1 + 2 j_2 + j_3 + j_4 =m\\
j_0\geq j_3 \\
		j_0,j_1,j_2,j_3,j_4 \in \mathbb{Z}_{\geq 0}
	}
	}
(1+j_3)(1+j_4)x_1^{j_1}x_2^{j_2}x_3^{j_3}x_4^{j_4}
\end{equation}
and
\begin{equation}\label{eq:ef2}
b_m=\sum_{
	\substack{
		j_0+j_1 + 2 j_2 + j_3 + j_4 =m\\
j_0 < j_3 \\
		j_0,j_1,j_2,j_3,j_4 \in \mathbb{Z}_{\geq 0}
	}
	}
(1+j_0)(1+j_4) x_1^{j_1}x_2^{j_2}x_3^{j_3}x_4^{j_4}
\end{equation}
One can verify the following without much difficulty :
\begin{equation}\label{eqn:efparta}
\sum_{m=0}^{\infty}a_mt^m=\frac{1}{\left(1-t\right)\left(1-t x_1\right) \left(1-t^2 x_2\right)\left(1-t^2 x_3\right)^2 \left(1-t x_4\right)^2}
\end{equation}
and
\begin{equation}\label{eqn:efpartb}
\sum_{m=0}^{\infty}b_mt^m=
\frac{x_3t}{
\left(1-t x_1\right) \left(1-t^2 x_2\right) \left(1-t x_3\right) \left(1-t^2 x_3\right)^2 \left(1-t x_4\right)^2}.
\end{equation}
By adding \eqref{eqn:efparta} and \eqref{eqn:efpartb}, we get the generating function of $\{P_m\}_{m=0}^{\infty}$, which finishes the proof.
\end{proof}
\begin{theorem}\label{thm:sumef}
Conjecture \ref{conj:sumef} holds if \eqref{eqn:sumef} is true for $m=0,\dots, \ell-1$, where 
$$
\ell=
\begin{cases}
6895, & \text{if $\mathfrak{g}$ is of type $E_6$ and $a=3$}\\
27217, & \text{if $\mathfrak{g}$ is of type $E_7$ and $a=2$}\\
143761, & \text{if $\mathfrak{g}$ is of type $E_8$ and $a=6$}\\
553, & \text{if $\mathfrak{g}$ is of type $F_4$ and $a=2$}.
\end{cases}
$$
\end{theorem}
\begin{proof}
Let us put $\lambda_1,\dots,\lambda_4$ as in \eqref{eqn:sumefla} in each case. Define
$P_m^{(a)}$ as the character of the RHS of \eqref{eqn:sumef} and $\mathcal{L}$ as follows :
\begin{equation}\label{eqn:efop}
\begin{aligned}
\mathcal{L}=& \left(
(1-\Delta) \prod_{\lambda\in \mathcal{O}(\lambda_1)}\left(1- e^{\lambda}\Delta \right)\prod_{\lambda\in \mathcal{O}(\lambda_2)}\left(1-e^{\lambda}\Delta ^2\right)\right)
\\
& \quad \times \left(\prod_{\lambda\in \mathcal{O}(\lambda_3)}\left(1-e^{\lambda}\Delta\right) \left(1-e^{\lambda}\Delta^2\right) \prod_{\lambda\in \mathcal{O}(\lambda_4)}\left(1- e^{\lambda}\Delta\right)^2 \right).
\end{aligned}
\end{equation}
Then Lemma \ref{lem:sumef} together with the same argument as in Proposition \ref{prop:lpsfL} implies $\mathcal{L}[P_m^{(a)}]=0$ and the order of $\mathcal{L}$  is $$\ell=1+|\mathcal{O}(\lambda_1)|+2|\mathcal{O}(\lambda_2)|+3|\mathcal{O}(\lambda_3)|+2|\mathcal{O}(\lambda_4)|.$$ 

On the other hand, we have shown that Theorem \ref{thm:main} holds with $\Lambda_a$ as given in Appendix \ref{sec:lambda}. In particular, $\{Q_m^{(a)}\}_{m=0}^{\infty}$ satisfies $\mathcal{L}_a[Q_m^{(a)}]=0$. We observe that $\mathcal{L}_a$ is a divisor of $\mathcal{L}$ as a polynomial in $\Delta$ :
$$
\begin{aligned}
\mathcal{L}_a=\left(
(1-\Delta)
\prod_{\lambda\in \mathcal{O}(\lambda_1)}\left(1- e^{\lambda}\Delta \right)
\prod_{\lambda\in \mathcal{O}(\lambda_3)}\left(1- e^{\lambda}\Delta \right)
\prod_{\lambda\in \mathcal{O}(\lambda_4)}\left(1- e^{\lambda}\Delta \right)
\right)
\end{aligned},
$$
which implies $\mathcal{L}[Q_m^{(a)}]=0$.

Thus if $P_m^{(a)}=Q_m^{(a)}$ for $m=0,\dots, \ell-1$, then $P_m^{(a)}=Q_m^{(a)}$ for each $m\in \mathbb{Z}_{\geq 0}$.
\end{proof}
\begin{remark}
Although this certainly reduces Conjecture \ref{conj:sumef} to a finite amount of computation involving the fermionic formula, the number $\ell$ is rather big and thus requires non-trivial computational effort. We will report the result of verification elsewhere.
\end{remark}
\subsection{A new proof of Theorem \ref{thm:sumg2}}
\begin{lemma}\label{lem:sumg2}
Let $\{P_m\}_{m=0}^{\infty}$ be the sequence given by
$$
P_m=\sum_{
\substack{
		j_0+3j_1 + j_2 = m\\
		j_0,j_1,j_2 \in \mathbb{Z}_{\geq 0}
	}
}
\left\{
\min\left(
	\Bigl \lfloor \frac{j_0}{3} \Bigr \rfloor,\,
	\Bigl \lfloor \frac{j_2-j_0}{3} \Bigr \rfloor + \Bigl \lfloor \frac{j_0}{3} \Bigr \rfloor
	\right)
+1 \right\}( j_1+1 )x_1^{j_1}x_2^{j_2}.
$$
We have
$$
\sum_{m=0}^{\infty}P_mt^m=\frac{1}{
\left(1-t^3\right) \left(1-t^3 x_1\right){}^2 \left(1-t x_2\right) \left(1-t^2 x_2\right)}.
$$
\end{lemma}
\begin{proof}
Let
$$
a_m=\sum_{
\substack{
		j_0+3j_1 + j_2 = m\\
		j_0,j_1,j_2 \in \mathbb{Z}_{\geq 0}
	}
}
\left\{
\min\left(
	\Bigl \lfloor \frac{j_0}{3} \Bigr \rfloor,\,
	\Bigl \lfloor \frac{j_2-j_0}{3} \Bigr \rfloor + \Bigl \lfloor \frac{j_0}{3} \Bigr \rfloor
	\right)+1
\right\}x_1^{j_1}x_2^{j_2}.
$$
We have $a_m=b_m+c_m+d_m$ where
$$
b_m=\sum_{
\substack{
		j_0+3j_1 + j_2 = m\\
		j_0,j_1,j_2 \in \mathbb{Z}_{\geq 0}
	}
}
\Bigl \lfloor \frac{j_0}{3} \Bigr \rfloor x_1^{j_1}x_2^{j_2},\quad 
c_m=\sum_{
\substack{
		j_0+3j_1 + j_2 = m\\
j_2\leq j_0 \\
		j_0,j_1,j_2 \in \mathbb{Z}_{\geq 0}
	}
}
	\Bigl \lfloor \frac{j_2-j_0}{3} \Bigr \rfloor
	x_1^{j_1}x_2^{j_2}, \quad
d_m=\sum_{
\substack{
		j_0+3j_1 + j_2 = m\\
		j_0,j_1,j_2 \in \mathbb{Z}_{\geq 0}
	}
}x_1^{j_1}x_2^{j_2}
.
$$
It is easy to check 
$$
\begin{aligned}
\sum_{m=0}^{\infty} b_mt^m&=\frac{t^3}{(1-t) \left(1-t^3\right) \left(1-t^3 x_1\right) \left(1-t x_2\right)},\\
\sum_{m=0}^{\infty} c_mt^m&=-\frac{t}{(1-t) \left(1-t^3\right) \left(1-t^3 x_1\right) \left(1-t^2 x_2\right)},\\
\sum_{m=0}^{\infty} d_mt^m&=\frac{1}{\left(1-t\right)\left(1-t^3 x_1\right) \left(1-t x_2\right)}.
\end{aligned}
$$
Therefore we get
$$
A(t,x_1,x_2)=\sum_{m=0}^{\infty} a_mt^m=\frac{1}{\left(1-t^3\right) \left(1-t^3 x_1\right) \left(1-t x_2\right) \left(1-t^2 x_2\right)},
$$
which implies
$$
\begin{aligned}
\sum_{m=0}^{\infty}P_mt^m&=(1+x_1\frac{\partial}{\partial x_1})A(t,x_1,x_2)\\
&=
\frac{1}{
\left(1-t^3\right) \left(1-t^3 x_1\right){}^2 \left(1-t x_2\right) \left(1-t^2 x_2\right)}.
\end{aligned}
$$
\end{proof}

\begin{proof}[Proof of Theorem \ref{thm:sumg2}]
Let $P_m^{(2)}$ be the character of the RHS of \eqref{eqn:sumg2}. From Lemma \ref{lem:sumg2} combined with the same argument as in Proposition \ref{prop:lpsfL}, we have $
\mathcal{L}[P_m^{(2)}]=0$
where 
$$
\begin{aligned}
\mathcal{L}&=\left(
(1-\Delta^3) \prod_{\lambda\in \mathcal{O}(\omega_1)}\left(1- e^{\lambda}\Delta^3 \right)^2 \prod_{\lambda\in \mathcal{O}(\omega_2)}\left(1-e^{\lambda}\Delta\right) \left(1-e^{\lambda}\Delta^2\right) 
\right).
\end{aligned}
$$
The order of $\mathcal{L}$ is $\ell=3+6|\mathcal{O}(\omega_1)|+3|\mathcal{O}(\omega_2)|=57$.

On the other hand, we have shown that Theorem \ref{thm:main} holds with $\Lambda_2$ and $\Lambda'_2$ as given in Appendix \ref{sec:lambda}. In particular, $\{Q_m^{(2)}\}_{m=0}^{\infty}$ satisfies $\mathcal{L}_2[Q_m^{(2)}]=0$, where
$$
\mathcal{L}_2=\left(
(1-\Delta^3) \prod_{\lambda\in \mathcal{O}(\omega_1)}\left(1- e^{\lambda}\Delta^3 \right) \prod_{\lambda\in \mathcal{O}(\omega_2)}\left(1-e^{\lambda}\Delta\right)\right).
$$ 
We observe that $\mathcal{L}_2$ is a divisor of $\mathcal{L}$ as a polynomial in $\Delta$, which implies $\mathcal{L}[Q_m^{(a)}]=0$ as well.

Thus if $P_m^{(a)}=Q_m^{(a)}$ for $0\leq m \leq \ell-1=56$, then $P_m^{(a)}=Q_m^{(a)}$ for each $m\in \mathbb{Z}_{\geq 0}$. Then it is not hard to check for $0\leq m \leq 56$ either $P_m^{(a)}=Q_m^{(a)}$ using the fermionic formula \eqref{eq:fer} or $P_m^{(a)}=R_{m}^{(a)}$, where $R_{m}^{(a)}$ is given in Proposition \ref{prop:QR}. This finishes the proof.
\end{proof}
\begin{remark}
In \cite{MR1745263}, there is one more conjectural lattice point summation formula, which is for the node $a=4$ in type $E_7$ and more complicated than \eqref{eqn:sumg2}. As we were not able to prove Theorem \ref{thm:main} for this node, we do not consider this here. We expect that formulas of the form \eqref{eq:stepp} exist for any $\mathfrak{g}$ and $a\in I$ and we will leave the search for them in the missing cases for future work.
\end{remark}

\section{Application (II) : dimension quasipolynomials of KR modules}\label{sec:app2}
In this section we prove that the dimension function $m\mapsto \dim W_m^{(a)}$ of the KR modules associated with $a\in I$ is a quasipolynomial in $m$ ; see Theorem \ref{dimWqp}. This is another consequence of Theorem \ref{thm:main}.

Let us recall the definition of quasipolynomials. Consider a function $p:\mathbb{Z}_{\geq 0} \to \mathbb{C}$ of the form $p(m)=c_{0;m}m^{d}+c_{1;m}m^{d-1}+\cdots +c_{d;m}$ where $d\in \mathbb{Z}_{\geq 0}$ and each $c_{i;m},\, 0\leq i \leq d$ is a complex-valued periodic function in $m$ and in particular, $c_{0;m}$ is non-zero. For simplicity, we will put $c_{i;m}=0$ if $i<0$ or $i>d$. We call $p$ a \textit{quasipolynomial} of \textit{degree} $d$ and $c_{i;m}$ its coefficients. The least common period of the coefficients is called the \textit{period} of $p$. Equivalently, $p$ is a quasipolynomial if there exist polynomials $p_0,\dots, p_{\ell-1}$ for some positive integer $\ell$ such that $p(m) = p_{j}(m)$ if $m \equiv j \pmod \ell$ for each $j=0,\dots, \ell-1$.

Let us collect here the key properties of $q_a(m) = \dim W_{m}^{(a)}$ that we will establish :
\begin{itemize}
\item $q_{a}$ is a quasipolynomial of degree $e_a$, of period dividing $t_a$ and with constant leading coefficient $c_{a;0}$ (Theorem \ref{dimWqp})
\item $q_{a}(m)=0$ if $m=-1,-2,\dots, -t_a h^{\vee}+1$ and $p_{a}(-t_a h^{\vee})\neq 0 $ (Proposition \ref{caformula})
\item $q_{a}(-m)=(-1)^{e_a}p_{a}(m-t_a h^{\vee})$ for $m\in \mathbb{Z}$ (Proposition \ref{EMrec})
\end{itemize}
where $e_a$ and $c_{a;0}$ are given in Proposition \ref{degcoef}.

\subsection{Quasipolynomial solutions of $Q$-systems}
In this subsection, we prove some properties of quasipolynomial solutions of $Q$-systems. For quasipolynomials, it is convenient to rewrite the $Q$-system \eqref{eq:Qsys} as follows.
\begin{definition}
Let $(q_a)_{a\in I}$ be a family of quasipolynomials such that
\begin{equation}\label{eq:poly-Qsys}
q_{a}(m)^2 - q_{a}({m+1})q_{a}({m-1}) = \prod_{b : C_{ab}< 0} \prod_{k=0}^{-C_{a b}-1}q_{b}\left(\left\lfloor\frac{C_{b a}m - k}{C_{a b}}\right\rfloor\right)
\end{equation}
for each $a\in I$ and $m\in \mathbb{Z}_{\geq 0}$. 
Then we call $(q_a)_{a\in I}$ a \textit{quasipolynomial solution of the $Q$-system}. As \eqref{eq:poly-Qsys} becomes $m$-independent relations among the coefficients of the quasipolynomials, the relation \eqref{eq:poly-Qsys} holds for any $a\in I$ and $m\in \mathbb{Z}$.
\end{definition}

Let us first prove two simple lemmas to handle the LHS of \eqref{eq:poly-Qsys} for quasipolynomials.
\begin{lemma}\label{pexp}
Let $p(m)=c_{0;m}m^{d}+c_{1;m}m^{d-1}+\cdots +c_{d;m}$ be a quasipolynomial of degree $d\geq 0$. Then the quasipolynomial $q_a(m)^2-q_a(m-1)q_a(m+1)$ has the following leading terms :
$$
\begin{aligned}
m^{2d} &(c_{0;m}^2-c_{0;m-1} c_{0;m+1})  +  m^{2d-1}(-c_{0;m+1} c_{1;m-1}+2 c_{0;m} c_{1;m}-c_{0;m-1} c_{1;m+1})\\
+ m^{2d-2} & (c_{1;m}^2-c_{1;m-1} c_{1;m+1}-c_{0;m+1} c_{1;m-1}+c_{0;m-1} c_{1;m+1} \\
 & + d c_{0;m-1} c_{0;m+1}-c_{0;m+1} c_{2;m-1}+2 c_{0;m} c_{2;m}-c_{0;m-1} c_{2;m+1})  + \cdots \end{aligned}.
$$
\end{lemma}
\begin{proof}
It follows from a straightforward computation.
\end{proof} 

\begin{lemma}\label{Amean}
Let $\{A_n\}_{n=0}^{\infty}$ be a complex-valued periodic sequence such that $A_{n+1}-2A_n+A_{n-1}$ is constant. Then $A_n$ is constant.
\end{lemma}
\begin{proof}
Let $A_{n+1}-2A_n+A_{n-1}=C$ for some $C\in \mathbb{C}$. Then any solution of this difference equation is of the form $A_n = A_0 + B n + C n^2/2$ for some $B\in \mathbb{C}$. For $A_n$ to be periodic in $n$, both $B$ and $C$ must be zero, which implies $A_n=A_0$.
\end{proof}

Now we prove that the degrees and the leading coefficients of quasipolynomials satisfying \eqref{eq:poly-Qsys} are essentially uniquely determined.
\begin{proposition}\label{degcoef}
Assume that $(q_a)_{a\in I}$ is a quasipolynomial solution of the $Q$-system such that the degree $e_a$ of $q_a$ is non-zero and the leading coefficient is a positive constant, say, $c_{a;0}$ for each $a\in I$.
Let $q_a(m)=c_{a;0}m^{e_a}+c_{a;1;m}m^{e_a-1}+\cdots +c_{a;e_a;m}$. Then the next two leading coefficients $c_{a;1;m}$ and $c_{a;2;m}$ are also independent of $m$. The degree of $q_a$ is given by
\begin{equation}\label{eaCartan}
e_a=2\sum_{b\in I}(C^{-1})_{ab},
\end{equation}
and the leading coefficient is uniquely determined by the system of equations
\begin{equation}\label{caCartan}
e_a\prod_{b\in I} (c_{b;0})^{C_{ab}}
 = \prod_{b:C_{ab}<0} \left(\frac{C_{ba}}{C_{ab}}\right)^{-C_{ab}e_b},\quad a\in I.
\end{equation}
\end{proposition}
\begin{proof}
As each term in the product of the RHS of \eqref{eq:poly-Qsys} is a quasipolynomial with constant leading coefficient, so is the product of them. Let us consider the coefficients of $q_a(m)^2-q_a(m-1)q_a(m+1)$ using Lemma \ref{pexp}. The coefficient $(c_{a;0}^2-c_{a;0} c_{a;0})$ of $m^{2e_a}$ clearly vanishes.
The coefficient of the potential leading term $m^{2e_a-1}$ is given as 
\begin{equation} \label{co1}
c_{a;0}(-c_{a;1;m-1}+2 c_{a;1;m}-c_{a;1;m+1}),
\end{equation}
which must be constant. As $c_{a;0}>0$, the second factor $(-c_{a;1;m-1}+2 c_{a;1;m}-c_{a;1;m+1})$ must be constant and then by Lemma \ref{Amean}, $c_{a;1;m}$ must also be constant. This implies 
\eqref{co1} is actually zero. Now we look at the coefficient of the next potential leading term $m^{2e_a-2}$ :
\begin{equation} \label{co2}
e_a c_{a;0}^2 +c_{a;0} (-c_{a;2;m-1}+2 c_{a;2;m}-c_{a;2;m+1}),
\end{equation}
which must be also constant. Then again by Lemma \ref{Amean}, $c_{a;2;m}$ must be independent of $m$, which turns \eqref{co2} into $e_a c_{a;0}^2$.

Now we express the degree $e_a$  of $q_a$ explicitly. By Lemma \ref{pexp} and \eqref{co2}, $q_a(m)^2-q_a(m-1)q_a(m+1)$ is a quasipolynomial whose leading term is $e_a c_{a;0}^2 m^{2e_a-2}$, which is clearly non-zero. We can also see the leading term of the RHS of \eqref{eq:poly-Qsys} is given by
$$
\prod_{b:C_{ab}<0} \left(c_{b;0}\left(\frac{C_{ba}}{C_{ab}}m\right)^{e_b}\right)^{-C_{ab}} = \left(\prod_{b:C_{ab}<0} \left(c_{b;0}\left(\frac{C_{ba}}{C_{ab}}\right)^{e_b}\right)^{-C_{ab}}\right)\times m^{\sum_{b: C_{ab}<0}-C_{ab}e_{b}}
$$
Hence, $(e_a)_{a\in I}$ and $(c_{a;0})_{a\in I}$ must satisfy the following systems of equations
$$
2e_{a}-2=\sum_{b: C_{ab}<0}-C_{ab}e_{b},\quad a\in I
$$
and
$$
e_ac_{a;0}^2
 = \prod_{b:C_{ab}<0} \left(c_{b;0}\left(\frac{C_{ba}}{C_{ab}}\right)^{e_b}\right)^{-C_{ab}}=\left(\prod_{b:C_{ab}<0} c_{b;0}^{-C_{ab}}\right)\left(\prod_{b:C_{ab}<0} (\frac{C_{ba}}{C_{ab}})^{-C_{ab}e_b}\right),\quad a\in I,
$$
from which our assertion follows.
\end{proof}
\begin{example}\label{g2ec}
Let consider the case when $\mathfrak{g}$ is of type $G_2$. The Cartan matrix and its inverse are 
$$
C = \left(
\begin{array}{cc}
 2 & -1 \\
 -3 & 2 \\
\end{array}
\right)
,\quad C^{-1} = 
\left(
\begin{array}{cc}
 2 & 1 \\
 3 & 2 \\
\end{array}
\right),
$$
respectively. From \eqref{eaCartan} we obtain $e_1 = 6$ and $e_2 = 10$. Then we can write down the equations \eqref{caCartan} for $c_{1;0}$ and $c_{2;0}$ :
$$
6 c_{1;0}^2 c_{2;0}^{-1} = 3^{10},\quad 10 c_{1;0}^{-3} c_{2;0}^2  =3^{-18}
$$
whose unique positive solution is $(c_{1;0}, c_{2;0}) =( 1/40 ,1/15746400)$.
\end{example}
A natural question from Proposition \ref{degcoef} is the following.
\begin{conjecture}\label{unique}
There exists a unique family $(q_a)_{a\in I}$ of quasipolynomial solution of the $Q$-system such that $\deg q_a$ is non-zero and the leading coefficient is a positive constant for each $a\in I$.
\end{conjecture}
We will see that the functions $(q_a)_{a\in I}$ given by $q_a(m) = \dim W_m^{(a)}$, is a quasipolynomial solution of the $Q$-system, which settles the existence part of the question. The uniqueness part remains conjectural.
\begin{remark}
Conjecture \ref{unique} can be regarded as the dimension analogue of the Kuniba-Nananishi-Suzuki conjecture \cite[Conjecture 14.2]{MR2773889}; see \cite{MR3282650, MR3043892, lee2013positivity} for more recent progresses. It roughly says that the unique positive solution of level $k\in \mathbb{Z}_{\geq 1}$ restricted $Q$-system, which is a finitized version of the $Q$-system, is given by the quantum dimensions of the Kirillov-Reshetikhin modules.
\end{remark}

\subsection{Dimensions of Kirillov-Reshetikhin modules in exceptional types}\label{qpEFG}
In this subsection we assume that $\mathfrak{g}$ is of exceptional type. The code used for the computations here can be found in the Github repository \cite{GHad341f7}, \url{https://github.com/chlee-0/KR-quasipolynomial}. For each $a\in I$, define $e_a$ as in \eqref{eaCartan}, $c_a=t_a (e_a+1-h^{\vee})$, and $\boldsymbol{h}_{a}=(h_{a;0},\dots, h_{a;c_a})$ as in \cite{GHad341f7}. For $m\in \mathbb{Z}$, define $q_a(m)$ by
\begin{equation}\label{efgqp}
q_a(m)=\sum_{k\geq 0}\binom{e_a+k}{e_a}h_{a;m-t_{a}k}=\sum_{\substack{j : 0\leq j\leq c_a \\ j \equiv m \pmod {t_a}}}\binom{e_a+(m-j)/t_a}{e_a}h_{a;j}.
\end{equation}
Then $q_a(m)$ is a quasipolynomial in $m$ of degree $e_a$ and period $t_a$. And each $q_a(t_am +k)$, $k=0,\dots, t_a-1$ is a polynomial in $m$. 

\begin{example}\label{qpG2}
In type $G_2$, we have $t_1 = 1$, $t_2 =3$, $e_1=6$, $e_2=10$, $h^{\vee}=4$, $c_1=3$, $c_2=21$, $\boldsymbol{h}_{1}=(1, 8, 8, 1)$ and $\boldsymbol{h}_{2}=(1, 7, 34, 122, 344, 803, 1581, 2683, 3952, 5100, 5785, 5785, 5100, 3952, 2683, 1581, 803, 344, 122, 34, 7, 1)$.
Thus $q_1(m), q_2(3m), q_2(3m+1)$ and $q_2(3m+2)$ are polynomials. Explicitly, we have
$$
\begin{aligned}
q_{1}(m) & = \frac{m^6}{40}+\frac{3 m^5}{10}+\frac{35 m^4}{24}+\frac{11 m^3}{3}+\frac{301 m^2}{60}+\frac{53 m}{15}+1,\\
q_{2}(3m) & = \frac{3 m^{10}}{800}+\frac{3 m^9}{40}+\frac{211 m^8}{320}+\frac{67 m^7}{20}+\frac{39133 m^6}{3600}+\frac{1409 m^5}{60} \\
& +\frac{98231 m^4}{2880}+\frac{1311 m^3}{40}+\frac{71491 m^2}{3600}+\frac{409 m}{60}+1, \\
q_{2}(3m+1) & = \frac{3 m^{10}}{800}+\frac{7 m^9}{80}+\frac{289 m^8}{320}+\frac{217 m^7}{40}+\frac{75523 m^6}{3600}+\frac{6539 m^5}{120} \\
& +\frac{276809 m^4}{2880}+\frac{27223 m^3}{240}+\frac{307801 m^2}{3600}+\frac{556 m}{15}+7, \\
q_{2}(3m+2) &= \frac{3 m^{10}}{800}+\frac{m^9}{10}+\frac{379 m^8}{320}+\frac{41 m^7}{5}+\frac{132223 m^6}{3600} +\frac{6667 m^5}{60}\\
& +\frac{661559 m^4}{2880}+\frac{9599 m^3}{30}+\frac{1032751 m^2}{3600}+\frac{8947 m}{60}+34. \\
\end{aligned}
$$
Note that this is consistent with Example \ref{g2ec}.
\end{example}

\begin{proposition}\label{excepionalpw}
Let $\mathfrak{g}$ be of exceptional type and $q_a$ be a quasipolynomial given in \eqref{efgqp} for each $a\in I$. Then $(q_a)_{a\in I}$ is a quasipolynomial solution of the $Q$-system and $q_a(m)=\dim W_m^{(a)}$ for all $m \in \mathbb{Z}_{\geq 0}$.
\end{proposition}
\begin{proof}
Let us first prove that $(q_a)$  is a quasipolynomial solution of the $Q$-system. Note that both sides of \eqref{eq:poly-Qsys} are quasipolynomials as each $q_a$ is a quasipolynomial. Thus, it is enough to check that the coefficients of these quasipolynomials are equal. This can be further reduced to comparison between the coefficients of some polynomials. In type $E_r, \, r=6,7,8$, these are polynomials. In type $F_4$, we can rewrite \eqref{eq:poly-Qsys} as
\begin{eqnarray*}
(q_{1}(m))^2 - q_{1}(m-1) q_{1}(m+1)
  &=& q_{2}(m), \\
(q_{2}(m))^2 - q_{2}(m-1) q_{2}(m+1)
  &=& q_{1}(m)q_{3}(2m),  \\
(q_{3}(2m))^2 - q_{3}(2m-1) q_{3}(2m+1)
      &=& (q_{2}(m))^2 q_{4}(2m), \\
(q_{3}(2m+1))^2 - q_{3}(2m) q_{3}(2m+2)
      &=& q_{2}(m) q_{2}(m+1) q_{4}(2m+1), \\
(q_{4}(2m))^2 - q_{4}(2m-1) q_{4}(2m+1)
&=& q_{3}(2m), \\
(q_{4}(2m+1))^2 - q_{4}(2m) q_{4}(2m+2)
&=& q_{3}(2m+1).
\end{eqnarray*}
In type $G_2$, \eqref{eq:poly-Qsys} is
\begin{eqnarray*}
(q_{1}(m))^2 -  q_{1}(m-1) q_{1}(m+1)
  &=& q_{2}(3m), \\
(q_{2}(3m))^2 - q_{2}(3m-1) q_{2}(3m+1)
  &=& (q_{1}(m))^3, \\
(q_{2}(3m+1))^2 - q_{2}(3m) q_{2}(3m+2)
  &=& (q_{1}(m))^2 q_{1}(m+1), \\
(q_{2}(3m+2))^2 - q_{2}(3m+1) q_{2}(3m+3)
  &=& q_{1}(m) (q_{1}(m+1))^2.
\end{eqnarray*}
We can see that each term in the above equations is a polynomial. Comparing the coefficients of these polynomials on both sides is a task which can be rather simply done by any computer algebra system.

We know that the family $\left(\dim W_m^{(a)}\right)_{a\in I, m\in \mathbb{Z}_{\geq 0}}$ of positive integers is a solution of the $Q$-system and so is $\left(q_a(m)\right)_{a\in I, m\in \mathbb{Z}_{\geq 0}}$. Then by using the $Q$-system recursion and the positivity of $\dim W_m^{(a)}$, we can conclude that $q_a(m) = \dim W_m^{(a)}$ for all $m\in \mathbb{Z}_{\geq 0}$ inductively, once we check $q_a(1) = \dim W_1^{(a)}$ for each $a\in I$. This is straightforward to check. See \cite{GHad341f7} for these computations.
\end{proof}

\subsection{Dimensions of Kirillov-Reshetikhin modules as quasipolynomials}
Now we assume $\mathfrak{g}$ is a simple Lie algebra of any type. For the rest of the paper, we define $q_a:\mathbb{Z}_{\geq 0}\to \mathbb{Z}_{\geq 0}$ by $q_a(m)=\dim W_m^{(a)}$ for each $a\in I$. Using Theorem \ref{thm:main} and Proposition \ref{excepionalpw}, we now prove the following :
\begin{theorem}\label{dimWqp}
Let $\mathfrak{g}$ be a simple Lie algebra. Then $q_a$ is a quasipolynomial of period dividing $t_a$, degree $e_a$ in \eqref{eaCartan}, and the leading coefficient $c_{a;0}$ is a positive constant given by \eqref{caCartan}.
\end{theorem}
\begin{proof}
In order to show that $q_a$ is a quasipolynomial of period dividing $t_a$, it suffices to show that the generating function of $\{q_a(m)\}_{m=0}^{\infty}$ can be written in the following form
\begin{equation}
\sum_{m=0}^{\infty}q_a(m)z^m=\frac{N_a(z)}{(1-z^{t_a})^{d_{a}}} \label{Ehrser}
\end{equation}
for some $N_a(z)\in \mathbb{Z}[z]$ and $d_{a}\in \mathbb{Z}_{\geq 0}$ such that $\deg N_a(z)<d_at_a$. When $\mathfrak{g}$ is not of type $E_7$ or $E_8$, it follows from Theorem \ref{thm:main} and Proposition \ref{prop:genL}. When $\mathfrak{g}$ is of type $E_7$ or $E_8$, we may use Proposition \ref{excepionalpw} directly. Thus each $q_{a}$ is a quasipolynomial of period dividing $t_a$.

Now let us prove that $e_a:=\deg q_a$ is actually non-zero, which might happen when $(1-z)^{d_a}$ is also a factor of $N_a(z)$ in \eqref{Ehrser}. Suppose that $e_a=0$. Then $q_a(m)$ must be periodic in $m$. However, as $\res W_m^{(a)}$ always contains a copy of $L(m\omega_a)$ as a $U_q(\mathfrak{g})$-submodule, $\{q_a(m)\}_{m=0}^{\infty}$ cannot be a bounded sequence, which is a contradiction. Therefore, $e_a$ is non-zero.

Let us show that the leading coefficient of $q_a$ is constant. Recall that the period of $q_a$ is a divisor of $t_a$, which is one of 1,2,3. When $t_a=1$, there is nothing to show as $q_a$ is a polynomial. When $t_a=3$, which can happen only when $\mathfrak{g}$ is of type $G_2$, we can see clearly  it in Example \ref{qpG2}. Thus we only need to consider the case when $q_a$ is of period $2$. Suppose $q_a(m) = c_{a;0;m}m^{e_a}+c_{a;1;m}m^{e_a-1}+\cdots +c_{a;e_a;m}$ and $c_{a;0;m}$ is of period $2$. Suppose that $c_{a;0;m}$ takes two different values, $C_0$ and $C_1$ when $m$ is even and odd, respectively. Both $C_0$ and $C_1$ must be non-negative as $q_a(m)$ is positive. From Proposition \ref{pexp}, we have
\begin{equation}\label{leadcoef}
q_a(m)^2-q_a(m-1)q_a(m+1)=(-1)^m(C_0^2-C_1^2)m^{2e_a}+\cdots.
\end{equation}
If $C_0>C_1$ (or $C_0<C_1$), then \eqref{leadcoef} becomes negative when $m$ is sufficiently large odd (even) integer. As $(q_a)_{a\in I}$ satisfies \eqref{eq:poly-Qsys} and hence the RHS of \eqref{eq:poly-Qsys} is always positive, this is a contradiction. 

Thus, $(q_a)_{a\in I}$ satisfies all the assumptions of Proposition \ref{degcoef}, which proves the theorem.
\end{proof}

For the rest of the section, we establish several properties of $q_a$ as a quasipolynomial. By Theorem \ref{dimWqp}, the generating function $E_{a}(z)$ of $\{q_a(m)\}_{m=0}^{\infty}$ is of the form
\begin{equation}\label{genea}
E_{a}(z)=\sum_{m=0}^{\infty}q_a(m)z^m=\frac{h_{a;0}+h_{a;1}z+\cdots +h_{a;c_a}z^{c_a}}{(1-z^{t_a})^{e_{a}+1}}
\end{equation}
for some $c_a\in \mathbb{Z}_{\geq 0}$ with $c_a<t_a (e_a+1)$, $h_{a;i}\in \mathbb{Z}$ for $i=0,1,\dots, c_a$ and $h_{a;c_a}\neq 0$. Note that \eqref{genea} implies $q_a$ is given by \eqref{efgqp} and vice versa. Let us call $\boldsymbol{h}_{a}=(h_{a;0},\dots, h_{a;c_a})$ the $h$-vector for $a\in I$.

In subsection \ref{qpEFG}, we found the exact form of \eqref{genea} when $\mathfrak{g}$ is of exceptional type. Let us recall some cases where we know the exact form of $E_{a}(z)$, when $\mathfrak{g}$ is of classical type.
\begin{prop}\label{dimWabcd}
Assume that $\mathfrak{g}$ is of classical type. We have
\begin{equation}\label{genE1}
E_{1}(z) = 
\begin{cases} 
 \frac{1}{(1-z)^{r+1}}, & \text{if $\mathfrak{g}$ is of type $A_r$}\\ 
 \frac{1+z}{(1-z)^{2r}}, & \text{if $\mathfrak{g}$ is of type $B_r$}\\ 
 \frac{(1+z)^{2r}}{(1-z^2)^{2r+1}}, & \text{if $\mathfrak{g}$ is of type $C_r$}\\ 
 \frac{1+z}{(1-z)^{2r-1}}, & \text{if $\mathfrak{g}$ is of type $D_r$}.
\end{cases}
\end{equation}
\end{prop}
\begin{proof}
This simply follows from the specialization of the generating function of the characters of the KR modules, given in \cite[Theorems 2.1, 2.4, 2.6, 2.8]{MR3342762}.
\end{proof}
For $E_1$ in \eqref{genE1}, let us compute the difference between the degrees of the numerator and the denominator; for type $A_r$, $B_r$, $C_r$ and $D_r$, we get $r+1$, $2r-1$, $2r+2$, and $2r-2$, respectively. We observe that these numbers are $t_1 h^{\vee}$, which turns out to be true in general. 

One way to determine $k_a: = t_a (e_a+1)-c_a$, which must be positive, is to examine the values of $q_a(m)$ when $m$ is negative; a general result of quasipolynomials is that $q_{a}(m)=0$ if $m=-1,-2,\dots, -k_a+1$ and $q_{a}(-k_a)\neq 0 $.

\begin{lemma}\label{stringABCD}
Let $\mathfrak{g}$ be of classical type. For each $a\in I$, we have $q_{a}(m)=0$ if $m=-1,-2,\dots, -t_a h^{\vee}+1$ and $q_{a}(-t_a h^{\vee})\neq 0$.
\end{lemma}
\begin{proof}
This holds for $a=1$ by Proposition \ref{dimWabcd}. Let us extend it to other $a\in I$. Recall that \eqref{eq:poly-Qsys} holds for any $m\in \mathbb{Z}$ as each $q_a$ is a quasipolynomial. For each $a\in I$, we easily obtain $q_a(-1)=0$ from \eqref{eq:poly-Qsys} as $q_a(0)=1$ and $q_a(1)\neq 0$. Then by the same argument as in \cite[Theorem 5.6]{lee2013positivity}, or by using \eqref{eq:poly-Qsys} several times, we can conclude $q_{a}(m)=0$ for any $a\in I$ and $m=-1,-2,\dots, -t_a h^{\vee}+1$.

Finally, \eqref{eq:poly-Qsys} gives
$$
\left(q_a(-t_a h^{\vee})\right)^2 =\prod_{b:C_{ab}<0} \left(q_b(-t_b h^{\vee}\right)^{-C_{a b}}, \,a\in I.
$$
As $q_{1}(-t_1h^{\vee}) \neq 0$, we obtain $q_{a}(-t_ah^{\vee}) \neq 0$ for each $a\in I$. This proves the lemma.
\end{proof}

\begin{prop}\label{caformula}
Let $\mathfrak{g}$ be a simple Lie algebra. We have (i) $c_a = t_a (e_a+1-h^{\vee})$ and (ii) $q_{a}(m)=0$ if $m=-1,-2,\dots, -t_a h^{\vee}+1$ and $q_{a}(-t_a h^{\vee})\neq 0 $.
\end{prop}
\begin{proof}
Since (i) and (ii) are equivalent, it is enough to prove one of the two. When $\mathfrak{g}$ is of classical type, Lemma \ref{stringABCD} proves (ii). When $\mathfrak{g}$ is of exceptional type, we can check (i) directly as in \cite{GHad341f7}.
\end{proof}

This allows us to write $c_a$ explicitly using the formula \eqref{eaCartan} for $e_a$. Therefore, to obtain the generating function $E_a$ in \eqref{genea}, it is now enough to find $\boldsymbol{h}_{a}$ using the values $q_a(m), \,0\leq m\leq c_a$ in general. It would be interesting to know if there is a way to express $\boldsymbol{h}_{a}$ in terms of Lie theoretic data as in \eqref{eaCartan} and \eqref{caCartan}.

Our next goal is the following reciprocity satisfied by $q_a$, which relates the values of $q_a$ at negative integers to those at positive integers :
\begin{equation}\label{recEM}
q_{a}(-m)=(-1)^{e_a}q_{a}(m-t_a h^{\vee}),\, m\in \mathbb{Z}.
\end{equation}
\begin{lemma}\label{EMreclem}
Suppose that $\boldsymbol{h}_{a}$ is symmetric, that is, $h_j = h_{c_a-j}$ for $0\leq j\leq c_a$. Then \eqref{recEM} holds.
\end{lemma}
\begin{proof}
Recall $$q_a(m)=\sum_{\substack{j : 0\leq j\leq c_a \\ j \equiv m \pmod {t_a}}}\binom{e_a+(m-j)/t_a}{e_a}h_{j}.$$
We can easily check that for $m\in \mathbb{Z}$,
$$
\begin{aligned}
\binom{e_a+\frac{-m-t_a h^{\vee}-(c_a-j)}{t_a}}{e_a} & = \binom{e_a+\frac{-m-t_a(e_a+1)+j}{t_a}}{e_a} \\
& = \binom{-1+\frac{-m+j}{t_a}}{e_a} \\
& = (-1)^{e_a} \binom{e_a+(m-j)/t_a}{e_a} 
\end{aligned}.
$$
As $h_j = h_{c_a-j}$, we obtain
$q_{a}(-m-t_a h^{\vee})= (-1)^{e_a}q_a(m)$.
\end{proof}

To remove the assumption on the symmetry of $\boldsymbol{h}_{a}$ in Lemma \ref{EMreclem}, let us prove a weaker version of the uniqueness of the quasipolynomial solutions of the $Q$-system.
\begin{lemma}\label{qpunique}
Let $\mathfrak{g}$ be of classical type. Let $(p_a)_{a\in I}$ be a quasipolynomial solution of the $Q$-system such that each $p_a$ is of positive degree, of period dividing $t_a$ and has a positive constant leading coefficient. If $p_1=q_1$, then $p_a=q_a$ for all $a\in I$.
\end{lemma}
\begin{proof}
Recall that $I=\{1,2,\dots, r\}$. We can find the biggest $r_0$ such that the Dynkin subdiagram with nodes $\{1,\dots, r_0\} \subseteq I$ is of type $A_{r_0}$. 

Let us assume that $2\leq l\leq r_0$. Suppose that we have shown that $q_a=p_a$ for $1\leq a \leq l-1$ and use induction on $l$. Then the $Q$-system \eqref{eq:poly-Qsys} for the node $a=l-1$ gives
$$
(p_{l-1}(m))^2 - p_{l-1}(m-1) p_{l-1}(m+1)  = p_{l-2}(m)p_{l}(m).
$$
As $p_{l-2}(m)$, which is equal to $q_{l-2}(m)$ by induction hypothesis, is positive for $m\geq 0$, we obtain
$$
\begin{aligned}
p_{l}(m)  & = \frac{(p_{l-1}(m))^2 - p_{l-1}(m-1) p_{l-1}(m+1)}{p_{l-2}(m)} \\
& = \frac{(q_{l-1}(m))^2 - q_{l-1}(m-1) q_{l-1}(m+1)}{q_{l-2}(m)}  \\
& = q_{l}(m)
\end{aligned}
$$
for any $m\in \mathbb{Z}_{\geq 0}$ and hence $p_{l}=q_{l}$. 

Now let us consider the remaining nodes in each type. The strategy is the same as above; write $p_l(m), r_0<l\leq r$ in terms of $p_{a}$ (and hence $q_{a}$) with $1\leq a \leq r_0$ using \eqref{eq:poly-Qsys} and show that $p_l(m)$ is uniquely determined for sufficiently large $m$, which determines $p_l$ itself uniquely. Them as $q_l(m)$ is determined by the same procedure, we can conclude $p_l=q_l$.

\textbf{type $A_r$} : As $r_0=r$, there is no remaining node and we are done.

\textbf{type $B_r$} : As $r_0 = r-1$, the only remaining node is $a=r$. Let us look at \eqref{eq:poly-Qsys} for $a=r-1$. We have two relations
$$
(p_{r-1}(m))^2 - p_{r-1}(m-1) p_{r-1}(m+1)  = p_{r-2}(m)p_{r}(2m),
$$
and
$$
(p_{r}(2m+1))^2 - p_{r}(2m) p_{r}(2m+2)  = p_{r-1}(m)p_{r-1}(m+1).
$$
The positivity of $p_{a}(m)$ for sufficiently large $m$, which is a consequence of positive degree and leading coefficient, uniquely determines the values of $p_r(2m)$ and $p_{r}(2m+1)$ and thus $p_r$ itself. Thus $p_a=q_a$ for all $a\in I$.

\textbf{type $C_r$} : As $r_0 = r-1$, the only remaining node is $a=r$. Consider \eqref{eq:poly-Qsys} for $a=r-1$ :
$$
(p_{r-1}(2m))^2 - p_{r-1}(2m-1) p_{r-1}(2m+1) = p_{r-2}(2m) (p_{r}(m))^2.
$$
Again, this uniquely determines the values of $p_{r}(m)$ and thus $p_r$ itself.

\textbf{type $D_r$} : As $r_0 = r-2$, the only two remaining nodes are $r-1$ and $r$. Then \eqref{eq:poly-Qsys} for $a=r-1$ and $a=r$ give
$$
\begin{aligned}
p_{r-2}(m) & = (p_{r-1}(m))^2 - p_{r-1}(m-1) p_{r-1}(m+1) \\
& = (p_{r}(m))^2 - p_{r}(m-1) p_{r}(m+1).
\end{aligned}
$$
As each $p_a$ is a quasipolynomial of degree dividing $t_a=1$, they are polynomials. It is an easy exercise to show that for a given polynomial $g$, there exists unique a polynomial $f$ (up to sign) such that
$$(f(m))^2 - f(m-1) f(m+1)  = g(m)$$
for all $m$, if there exists any. This implies $p_{r-1}=p_{r}=q_{r-1}=q_{r}$.
\end{proof}

\begin{proposition}\label{EMrec}
Let $\mathfrak{g}$ be a simple Lie algebra. For any $a\in I$, \eqref{recEM} holds. 
\end{proposition}
\begin{proof}
Let $p_a(m): = (-1)^{e_a}q_{a}(-m-t_a h^{\vee})$ for $m\in \mathbb{Z}$. Then $p_a$ is a quasipolynomial of degree $e_a$ with positive constant leading coefficient and $(p_a)_{a\in I}$ is a quasipolynomial solution of the $Q$-system.

Assume that $\mathfrak{g}$ is of classical type. Recall that from Proposition \ref{dimWabcd}, we have the symmetry of the $h$-vector for $a=1$. Thus, $p_1=q_1$ by Lemma \ref{EMreclem}. By applying Lemma \ref{qpunique}, we can conclude that $p_a = q_a$ for all $a\in I$.

Assume that $\mathfrak{g}$ is of exceptional type. We can check the symmetry of the $h$-vector of $q_a$ for each $a\in I$ directly using \cite{GHad341f7}. And hence by Lemma \ref{EMreclem}, $p_a = q_a$ for all $a\in I$.
\end{proof}

\begin{cor}\label{genEMrec}
For any $a\in I$, $E_{a}(1/z)=(-1)^{e_a+1}z^{t_a h^{\vee}}E_{a}(z)$.
\end{cor}
\begin{proof}
For any quasipolynomial $Q$, we have $R_Q^{+}(z)+R_Q^{-}(z)=0$, where $R_Q^{+}(z):=\sum_{m\geq 0}Q(m)z^m$ and $R_Q^{-}(z):=\sum_{m<0}Q(m)z^m$ as rational functions \cite[Exercise 4.7.]{MR3410115}. Thus $E_a(z)= -\sum_{m<0}^{\infty}q_a(m)z^{m}$ and $E_a(1/z)= -\sum_{m<0}^{\infty}q_a(m)z^{-m}$. Then
$$
\begin{aligned}
E_{a}(1/z) & = -\sum_{m>0}^{\infty}q_a(-m)z^{m} = -(-1)^{e_a} z^{t_a h^{\vee}}\sum_{m>0}^{\infty}q_a(m-t_a h^{\vee})z^{m-t_a h^{\vee}} \\
& = (-1)^{e_a+1} z^{t_a h^{\vee}} E_a(z),
\end{aligned}
$$
where we have used Proposition \ref{EMrec} to obtain the second equality.
\end{proof}

\begin{cor}
For any $a\in I$, $\boldsymbol{h}_{a}$ is symmetric.
\end{cor}

\begin{proof}
Note that
$$
\begin{aligned}
E_{a}(1/z) & = \frac{h_0+h_1z^{-1}+\cdots +h_{c_a}z^{-c_a}}{(1-z^{-t_a})^{e_{a}+1}} \\
& = \frac{(h_0+h_1z^{-1}+\cdots +h_{c_a}z^{-c_a})z^{t_a(e_{a}+1)}}{(1-z^{-t_a})^{e_{a}+1}z^{t_a(e_{a}+1)}} \\
& = \frac{(h_0+h_1z^{-1}+\cdots +h_{c_a}z^{-c_a})z^{c_a+t_ah^{\vee}}}{(z^{t_a}-1)^{e_{a}+1}} \\
& = \frac{(-1)^{e_a+1}z^{t_a h^{\vee}}(h_0z^{c_a}+h_1z^{c_a-1}+\cdots +h_{c_a}z^{c_a-c_a})}{(1-z^{t_a})^{e_{a}+1}}.
\end{aligned}
$$
By Corollary \ref{genEMrec}, the last line of the above must be equal to $(-1)^{e_a+1}z^{t_a h^{\vee}}E_{a}(z)$, which implies $h_{a;j} = h_{a;c_a -j}$ for $j=0,\dots, c_a$.
\end{proof}
We leave the following properties of the $h$-vector as a conjecture.
\begin{conjecture}\label{hconj}
The $h$-vector $\boldsymbol{h}_{a}= (h_{a;0},\dots, h_{a;c_a})$ satisfies the following properties :
\begin{itemize}
\item (positivity) $h_{a;j}$ is positive for $j=0,\dots, c_a$
\item (unimodality) $h_{a;0}\leq \cdots \leq h_{a;\lfloor c_a/2 \rfloor}$
\item (log-concavity) $h_{a;j}^2 -h_{a;j-1}h_{a;j+1}\geq 0$ for $j=1,\dots, c_a-1$
\end{itemize}
\end{conjecture}
When $\mathfrak{g}$ is of exceptional type, we have verified all these properties in \cite{GHad341f7}. In classical types, we have experimentally checked them for $\mathfrak{g}$ of rank up to 8.

Let us give the final remark on $q_a$. The reciprocity \eqref{recEM} satisfied by $q_a$ and the (conjectural) positivity of $\boldsymbol{h}_{a}$ lead us to speculate that the quasipolynomial $q_a$ is an object belonging to Ehrhart theory \cite{MR3410115}. Ehrhart theory is the study of lattice points inside rational polytopes. A \textit{rational polytope} $\mathcal{P}\subset \mathbb{R}^n$ is the convex hull of finitely many points in $\mathbb{Q}^n$. The \textit{lattice point enumerator} of $\mathcal{P}$ is defined to be
$$
L_{\mathcal{P}}(m):=\#(m \mathcal{P} \cap \mathbb{Z}^n) 
$$
for $m\in\mathbb{Z}_{>0}$.

Ehrhart fundamental theorem \cite{MR0130860} states that $L_{\mathcal{P}}(m)$ is a quasipolynomial of degree $e :=\dim \mathcal{P}$. Let $t$ be the smallest positive integer such that $t\mathcal{P}$ has vertices in $\mathbb{Z}^n$. Stanley's non-negativity theorem \cite{MR593545} says that the generating function of $L_{\mathcal{P}}(m)$ can be written as a rational function
$$
E_{\mathcal{P}}(z)=\sum_{m=0}^{\infty}L_{\mathcal{P}}(m)z^m=\frac{h_{0}+h_{1}z+\cdots +h_{t(e+1)-1}z^{t(e+1)-1}}{(1-z^{t})^{e +1}}
$$
where each $h_i$ is a non-negative integer. We have
$$
L_{\mathcal{P}}(-m) = (-1)^{e} L_{\mathcal{P}^{\circ}}(m) 
$$
where $\mathcal{P}^{\circ}$ denotes the interior of $\mathcal{P}$ and this called Ehrhart-Macdonald reciprocity. When there exists a positive integer $k$ such that $L_{\mathcal{P}^{\circ}}(1)=\cdots= L_{\mathcal{P}^{\circ}}(k-1)=0$ and $L_{\mathcal{P}^{\circ}}(m) = L_{\mathcal{P}}(m-k)$ for all $m>k$, $\mathcal{P}$ is called Gorenstein of index $k$. In such a case, we obtain $L_{\mathcal{P}}(-m) = (-1)^{e} L_{\mathcal{P}}(m-k)$.

Thus we propose the following as a way to understand $q_a$ from the viewpoint of Ehrhart theory or to attack the positivity in Conjecture \ref{hconj} :
\begin{conjecture}\label{KRpolytope}
Let $\mathfrak{g}$ be a simple Lie algebra. For each $a\in I$, there exists a rational polytope $\mathcal{P}^{(a)}$ of dimension $e_a$ such that the set of lattice points of $m\mathcal{P}^{(a)}$ has an affine crystal structure, isomorphic to the KR crystal associated with $W_m^{(a)}$.
\end{conjecture}
For an explicit construction of $\mathcal{P}^{(a)}$ in type $A$, see \cite{MR3102176}.

\section*{Acknowledgements}
The author wishes to thank Masoud Kamgarpour for helpful discussions on Weyl groups. He is  also grateful to Masato Okado, Travis Scrimshaw and Per Alexandersson for helpful comments on an earlier version of this paper.
\appendix
\section{Weights of fundamental representations}\label{sec:chwo}
For each $\mathfrak{g}$ and $a\in I$, we will describe $\Omega \left(L(\omega_a)\right)$. It is enough to give $\Omega \left(L(\omega_a)\right)\cap P^+$. 
Note that $\lambda \in \Omega \left(L(\omega_a)\right)\cap P^+$ if and only if $\omega_a\succeq \lambda$; see, for example, Proposition 21.3 of \cite{MR499562}. Thus it can be easily found once we write each $\omega_a$ in terms of simple roots and we will not write down the detailed calculation for this here. For exceptional types, this also can be found by using the Freudenthal multiplicity formula \cite{MR499562}, which is implemented in {\sc Sage}~\cite{sage}. 

For the reader's convenience we include all the nodes in type $E_7$ and $E_8$, which are not covered in Theorem \ref{thm:main} and thus have not been used in this paper; those nodes are indicated by the bold letters in the tables below. Note that again $\omega_0=0$.

\textbf{type} $A_r$ :
$$\Omega \left(L(\omega_a)\right)\cap P^+=\{\omega_a\} \text{ for any } a\in I
$$

\textbf{type} $B_r$ :
$$\Omega \left(L(\omega_a)\right)\cap P^+=
\begin{cases} 
\{\omega_0,\omega_1,\dots,\omega_{a}\}, & \text{if $1\leq a\leq r-1$}\\ 
\{\omega_a\}, & \text{if $a=r$}
\end{cases}
$$

\textbf{type} $C_r$ :
$$
\Omega \left(L(\omega_a)\right)\cap P^+=
\begin{cases} 
\{\omega_0, \omega_2,\dots,\omega_{a}\} & \text{if $a$ is even}\\ 
\{\omega_1, \omega_3, \dots, \omega_{a}\} & \text{if $a$ is odd}
\end{cases}
$$

\textbf{type} $D_r$ :
$$
\Omega \left(L(\omega_a)\right)\cap P^+=
\begin{cases} 
\{\omega_0, \omega_2, \dots , \omega_{a}\}, & \text{if $1\leq a\leq r-2$ and $a$ is even}\\ 
\{\omega_1, \omega_3, \dots , ,\omega_{a}\}, & \text{if $1\leq a\leq r-2$ and $a$ is odd}\\
\{\omega_a\}, & \text{if $a=r-1,r$}
\end{cases}
$$

\textbf{type} $E_6$:
$$
\begin{array}{l|l}
 a & \text{elements of }\Omega \left(L(\omega_a)\right)\cap P^+ \\
\hline
 1 & \omega _1 \\
 2 & \omega _2,\omega _5 \\
 3 & 0,\omega _3,\omega _1+\omega _5,\omega _6 \\
 4 & \omega _1,\omega _4 \\
 5 & \omega _5 \\
 6 & 0,\omega _6 \\
\end{array}
$$

\textbf{type} $E_7$:
$$
\begin{array}{l|l}
 a & \text{elements of }\Omega \left(L(\omega_a)\right)\cap P^+ \\
\hline
 1 & 0,\omega _1 \\
 2 & 0,\omega _1,\omega _2,\omega _5 \\
 3 & 0,\omega _1,2 \omega _1,\omega _2,\omega _3,\omega _5,\omega _1+\omega _5,2 \omega _6,\omega _6+\omega _7 \\
 \textbf{4} & \omega _4,\omega _6,\omega _1+\omega _6,\omega _7 \\
 5 & 0,\omega _1,\omega _5 \\
 6 & \omega _6 \\
 7 & \omega _6,\omega _7 \\
\end{array}
$$

\textbf{type} $E_8$:
$$
\begin{array}{l|l}
 a & \text{elements of }\Omega \left(L(\omega_a)\right)\cap P^+ \\
\hline
 1 & 0,\omega _1,\omega _7 \\
 2 & 0,\omega _1,\omega _2,\omega _5,\omega _6,\omega _7,2 \omega _7,\omega _1+\omega _7,\omega _8 \\
 \textbf{3} & \text{see below} \\
 \textbf{4} & 0,\omega _1,2 \omega _1,\omega _2,\omega _4,\omega _5,\omega _6,\omega _1+\omega _6,\omega _7,2 \omega _7,\omega _1+\omega _7,\omega _6+\omega _7,\omega _8,\omega _7+\omega _8 \\
 \textbf{5} & 0,\omega _1,\omega _5,\omega _6,\omega _7,2 \omega _7,\omega _1+\omega _7,\omega _8 \\
 6 & 0,\omega _1,\omega _6,\omega _7 \\
 7 & 0,\omega _7 \\
 \textbf{8} & 0,\omega _1,\omega _6,\omega _7,\omega _8 \\
\end{array}
$$
The elements of $\Omega \left(L(\omega_3)\right)\cap P^+$ are 
$$
\begin{array}{llllllll}
 0, & \omega _3, & \omega _6, & 2 \omega _1, & \omega _1+\omega _6, & \omega _1+\omega _7, & \omega _5+\omega _7, & \omega _1+\omega _8, \\
 \omega _1, & \omega _4, & \omega _7, & \omega _1+\omega _5, & 2 \omega _7, & 2 \omega _1+\omega _7, & \omega _6+\omega _7, & \omega _6+\omega _8, \\
 \omega _2, & \omega _5, & \omega _8, & 2 \omega _6, & 3 \omega _7, & \omega _2+\omega _7, & \omega _1+2 \omega _7, & \omega _7+\omega _8. \\
\end{array}
$$

\textbf{type} $F_4$
$$
\begin{array}{c|c}
 a & \text{elements of }\Omega \left(L(\omega_a)\right)\cap P^+ \\
\hline
 1 & 0,\omega _1,\omega _4 \\
 2 & 0,\omega _1,\omega _2,\omega _3,\omega _4,2 \omega _4,\omega _1+\omega _4 \\
 3 & 0,\omega _1,\omega _3,\omega _4 \\
 4 & 0,\omega _4 \\
\end{array}
$$

\textbf{type} $G_2$
$$
\begin{array}{c|c}
 a & \text{elements of }\Omega \left(L(\omega_a)\right)\cap P^+ \\
\hline
 1 & 0,\omega _1,\omega _2 \\
 2 & 0,\omega _2 \\
\end{array}
$$
\section{Description of $\Lambda_a$ and $\Lambda'_a$}\label{sec:lambda}
For each $\mathfrak{g}$ and $a\in I$, we will describe $\Lambda_a$ and $\Lambda'_a$ used in the proof of Theorem \ref{thm:main}. As these are $W$-invariant subsets of $P$ and hence disjoint unions of $W$-orbits, it is enough to give $\Lambda_a\cap P^+$ and $\Lambda'_a\cap P^+$. When $\Lambda'_a\cap P^+$ is not specified, it simply means $\Lambda'_a=\emptyset$. In simply-laced types, $\Lambda_a$ is defined to be $\Omega \left(L(\omega_a)\right)$ in Appendix \ref{sec:chwo} and thus we do not repeat it here. We note that this is a refined version of the tables on the orders of difference operators $\mathcal{L}_a$ given in Appendix A of \cite{MR3342762}.

\textbf{type} $B_r$:
$$
\Lambda_a\cap P^+
=
\begin{cases} 
\{\omega_0, \omega_2,\dots, \omega_{a}\}, & \text{if $1\leq a\leq r-1$ and $a$ is even}\\ 
\{\omega_1, \omega_3,\dots, \omega_{a}\}, & \text{if $1\leq a\leq r-1$ and $a$ is odd} \\ 
\{\omega_a\}, & \text{if $a=r$}
\end{cases}
$$
and
$$
\Lambda_a' \cap P^+=
\begin{cases} 
\{\omega_0, \omega_2, \dots,\omega_{r-2}\} , & \text{if $a=r$ and $a$ is even}\\ 
\{\omega_1, \omega_3, \dots, \omega_{r-2}\}, & \text{if $a=r$ and $a$ is odd}
\end{cases}
$$

\textbf{type} $C_r$:
$$
\Lambda_a\cap P^+=\{\omega_a\} \text{ for any } a\in I
$$
and
$$
\Lambda'_a\cap P^+ 
=
\{\omega_0, \omega_1, \dots, \omega_{a-1}\},\, \text{if $1\leq a\leq r-1$}
$$

\textbf{type} $F_4$:
$$
\begin{array}{c|c|c}
 a &  \text{elements of }\Lambda_a\cap P^+ &  \text{elements of }\Lambda'_a\cap P^+ \\
\hline
 1 & 0,\omega _1 & \cdot \\
 2 & 0,\omega _1,\omega _2,2 \omega _4 & \cdot \\
 3 & \omega _3 & 0,\omega _1,2 \omega _1,\omega _2,2 \omega _4,\omega _1+2 \omega _4 \\
 4 & \omega _4 & 0,\omega _1 \\
\end{array}
$$

\textbf{type} $G_2$:
$$
\begin{array}{c|c|c}
 a &  \text{elements of }\Lambda_a\cap P^+ &  \text{elements of }\Lambda'_a\cap P^+ \\
\hline
 1 & 0,\omega _1 & \cdot \\
 2 & \omega _2 & 0,\omega _1 \\
\end{array}
$$
\section{Solution of $Q$-system of type $G_2$}\label{sec:appg2}
Here we give an explicit description of $c_{\lambda}\in \mathbb{Q}(y_1,y_2)$ for $\lambda\in \Lambda_1\coprod \Lambda_2$ and $t_{\lambda,j}\in \mathbb{Q}(y_1,y_2)$ for $\lambda\in \Lambda'_2,j\in \{0,1,2\}$ used in Lemma \ref{lem:g2}. Recall that $W$ acts on the field $\mathbb{Q}(y_1,y_2)$ as mentioned in Introduction. We will write only $c_0, c_{\omega_1},c_{\omega_2}$ and $t_{0,j}, t_{\omega_1,j}$ for $j\in \{0,1,2\}$ explicitly below. We define all the other elements by the rule $w(c_{\lambda})=c_{w(\lambda)}$ and $w(t_{\lambda,j})=t_{w(\lambda),j}$ for $w\in W$. As each $\Lambda_1, \Lambda_2$ and $\Lambda'_2$ is a disjoint union of $W$-orbits, the following gives all the necessary information :
$$
\begin{aligned}
c_0& =-\frac{2 y_1^3 y_2^4 \left(y_2^2+y_2+y_1\right) \left(y_2^2+y_1 y_2+y_1\right)}{\left(y_1-1\right){}^2 \left(y_1-y_2^3\right){}^2 \left(y_1^2-y_2^3\right){}^2}, \\
c_{\omega_1}&=-\frac{y_1^5 y_2^4}{\left(y_1-1\right){}^2 \left(y_1-y_2\right) \left(y_2-1\right) \left(y_1-y_2^3\right) \left(y_1^2-y_2^3\right)}, \\
c_{\omega_2}&=\frac{y_1^4 y_2^{13}}{\left(y_1-1\right){}^3 \left(y_1-y_2\right) \left(y_2-1\right){}^2 \left(y_1-y_2^2\right) \left(y_1-y_2^3\right){}^3},
\end{aligned}
$$
$$
\begin{aligned}
t_{0,0}&=-\frac{2 y_1^4 y_2^6 \left(y_2^8-y_1 y_2^7-y_2^7+y_1^2 y_2^6+y_1 y_2^6+y_2^6+y_1^2 y_2^5+y_1 y_2^5-y_1^3 y_2^4-6 y_1^2 y_2^4\right)}{\left(y_1-1\right){}^2 \left(y_1-y_2\right){}^2 \left(y_2-1\right){}^2 \left(y_1-y_2^2\right){}^2 \left(y_1-y_2^3\right){}^2 \left(y_1^2-y_2^3\right){}^2} \\
& \quad 
-\frac{2 y_1^4 y_2^6 \left(-y_1 y_2^4+y_1^3 y_2^3+y_1^2 y_2^3+y_1^4 y_2^2+y_1^3 y_2^2+y_1^2 y_2^2-y_1^4 y_2-y_1^3 y_2+y_1^4\right)}{\left(y_1-1\right){}^2 \left(y_1-y_2\right){}^2 \left(y_2-1\right){}^2 \left(y_1-y_2^2\right){}^2 \left(y_1-y_2^3\right){}^2 \left(y_1^2-y_2^3\right){}^2},
\\
t_{0,1}&=-\frac{2 y_1^4 y_2^7 \left(y_1 y_2^6+y_2^6-y_1 y_2^5-y_1^2 y_2^4-y_1 y_2^4+y_1^3 y_2^3\right)}{\left(y_1-1\right){}^2 \left(y_1-y_2\right){}^2 \left(y_2-1\right){}^2 \left(y_1-y_2^2\right){}^2 \left(y_1-y_2^3\right){}^2 \left(y_1^2-y_2^3\right){}^2}\\
&\quad
-\frac{2 y_1^4 y_2^7 \left(y_1 y_2^3-y_1^3 y_2^2-y_1^2 y_2^2-y_1^3 y_2+y_1^4+y_1^3\right)}{\left(y_1-1\right){}^2 \left(y_1-y_2\right){}^2 \left(y_2-1\right){}^2 \left(y_1-y_2^2\right){}^2 \left(y_1-y_2^3\right){}^2 \left(y_1^2-y_2^3\right){}^2},
\\
t_{0,2}&=t_{0,1},
\end{aligned}
$$
$$
\begin{aligned}
t_{\omega _1,0}&=-\frac{y_1^6 y_2^6 \left(y_2^8+4 y_1 y_2^6+2 y_1^2 y_2^5+2 y_1 y_2^5+9 y_1^2 y_2^4+2 y_1^3 y_2^3+2 y_1^2 y_2^3+4 y_1^3 y_2^2+y_1^4\right)}{\left(y_1-1\right){}^2 \left(y_1-y_2\right) \left(y_2-1\right) \left(y_1-y_2^3\right){}^3 \left(y_1^2-y_2^3\right){}^3},
\\
t_{\omega _1,1}&=-\frac{y_1^7 y_2^7 \left(2 y_2^6+3 y_2^5+6 y_1 y_2^4+y_1^2 y_2^3+4 y_1 y_2^3+6 y_1^2 y_2^2+3 y_1^2 y_2+2 y_1^3\right)}{\left(y_1-1\right){}^2 \left(y_1-y_2\right) \left(y_2-1\right) \left(y_1-y_2^3\right){}^3 \left(y_1^2-y_2^3\right){}^3},
\\
t_{\omega _1,2}&=-\frac{y_1^7 y_2^7 \left(2 y_2^6+3 y_1 y_2^5+6 y_1 y_2^4+4 y_1^2 y_2^3+y_1 y_2^3+6 y_1^2 y_2^2+3 y_1^3 y_2+2 y_1^3\right)}{\left(y_1-1\right){}^2 \left(y_1-y_2\right) \left(y_2-1\right) \left(y_1-y_2^3\right){}^3 \left(y_1^2-y_2^3\right){}^3}.
\end{aligned}
$$
\bibliographystyle{plain}
\bibliography{linQ}
\end{document}